\documentclass[10pt,reqno]{amsart}

\usepackage[english]{babel}

\usepackage[a4paper,top=2cm,bottom=2cm,left=2.2cm,right=2.2cm]{geometry}

\usepackage{amsmath}
\usepackage{amssymb}
\usepackage{amsthm}
\usepackage{graphicx}
\usepackage{xcolor}
\usepackage{tikz-cd}
\usepackage[colorlinks=true, allcolors=blue]{hyperref}
\usepackage{mathtools}
\usepackage{tcolorbox}
\usepackage{etoolbox}
\usepackage{enumitem}
\usepackage{moreenum}
\usepackage{orcidlink}

\providecommand\given{}
\newcommand\SetSymbol[1][]{%
	\nonscript\:#1\vert
	\allowbreak
	\nonscript\:
	\mathopen{}}
\DeclarePairedDelimiterX\Set[1]\{\}{%
	\renewcommand\given{\SetSymbol[\delimsize]}
	#1
}

\DeclarePairedDelimiterXPP{\Generate}[1]{}{\langle}{\rangle}{_\mathcal{R}}{%
	\renewcommand\given{\SetSymbol[\delimsize]}
	#1
}

\DeclarePairedDelimiter{\abs}{\lvert}{\rvert}

\DeclarePairedDelimiterXPP{\uNorm}[2]{}{\lVert}{\rVert}{_{#2}}{
\ifblank{#1}{\: .\:}{#1}}

\newcommand{\fa}{\;\forall\,}
\newcommand{\ex}{\;\exists\,}

\newcommand{\R}[1][]{\ifstrempty{#1}{\mathbb{R}}{\mathbb{R}^{#1}}}
\newcommand{\C}[1][]{\ifstrempty{#1}{\mathbb{C}}{\mathbb{C}^{#1}}}
\newcommand{\N}{\mathbb{N}}

\newcommand{\bM}{\mathbf{M}}
\newcommand{\bN}{\mathbf{N}}
\newcommand{\fM}{\mathfrak{M}}
\newcommand{\E}{\mathcal{E}}
\newcommand{\D}{\mathcal{D}}

\newcommand{\crb}{\mathcal{V}}
\newcommand{\atlas}{\mathcal{A}}
\newcommand{\poln}{\mathcal{P}_n}

\newcommand{\analytic}{\mathcal{C}^{\omega}}

\newcommand{\Rou}[2][]{\ifstrempty{#1}{\mathcal{E}^{\{#2\}}}{\mathcal{E}^{\{#2\}}(#1)}}
\newcommand{\Beu}[2][]{\ifstrempty{#1}{\mathcal{E}^{(#2)}}{\mathcal{E}^{(#2)}(#1)}}

\newcommand{\sheaf}[2][]{%
  \ifstrempty{#1}{%
   \ifstrempty{#2}{ \mathcal{R}}{\mathcal{R}_{#2}}
  }{%
   \ifstrempty{#2}{\mathcal{R}(#1)}{\mathcal{R}_{#2}(#1)} 
  }%
}
\newcommand{\module}[1][]{%
\ifstrempty{#1}{ \mathcal{N}}{\mathcal{N}(#1)}}

\newcommand{\Distr}[2][]{\ifstrempty{#1}{%
\ifstrempty{#2}{\mathcal{D}^\prime}{\mathcal{D}^\prime_{#2}}%
}{\mathcal{D}^\prime_{#2}(#1)}}

\newcommand{\pot}[2][]{\ifstrempty{#1}{\ifstrempty{#2}{\mathfrak{p}^{sc}}{\mathfrak{p}^{sc}_{#2}}}%
{\mathfrak{p}^{sc}_{#2}(#1)}}

\newcommand{\VFS}[1][]{\ifstrempty{#1}{\mathfrak{X}}{\mathfrak{X}(#1)}}

\DeclareMathOperator{\WFR}{WF_{\mathcal{R}}}
\DeclareMathOperator{\WF}{WF}
\DeclareMathOperator{\real}{Re}
\DeclareMathOperator{\imag}{Im}
\DeclareMathOperator{\supp}{supp}
\DeclareMathOperator{\singsupp}{sing\, supp}
\DeclareMathOperator{\Char}{Char}
\DeclareMathOperator{\spanc}{span}
\DeclareMathOperator{\Id}{Id}

\DeclareMathOperator{\GL}{GL}

\DeclareMathOperator{\Fl}{Fl}

\DeclareMathOperator{\Imag}{imag}

\theoremstyle{plain}
\newtheorem{Thm}{Theorem}[section]
\newtheorem{Prop}[Thm]{Proposition}
\newtheorem{Lem}[Thm]{Lemma}
\newtheorem{Cor}[Thm]{Corollary}

\theoremstyle{definition}
\newtheorem{Def}[Thm]{Definition}

\theoremstyle{remark}
\newtheorem{Rem}[Thm]{Remark}

\title{Sheaves of ultradifferentiable functions}
\author{Stefan F\"urd\"os\;\orcidlink{0000-0003-2612-5349} }
\address{Faculty of Mathematics, University of Vienna, Oskar-Morgenstern-Platz 1, 1090 Wien, Austria}
\email{stefan.fuerdoes@univie.ac.at}

\subjclass[2020]{Primary 26E10, 35A18, 35A27; %
Secondary 32V20, 35B65, 58J99} 
\keywords{sheaves of ultradifferentiable functions, (non-)quasianalyticity, wave front set, regularity}

\begin{document}

\begin{abstract}
An abstract theory of ultradifferentiable sheaves is developed.
Moreover, various applications to the theory of linear partial differential equations, 
differential geometry and, in particular, CR geometry are discussed.
\end{abstract}
\maketitle
\section{Introduction}
In recent times there has been a resurgent interest 
in the study and applications of ultradifferentiable classes, 
i.e.~subalgebras of the algebra of smooth functions.

Usually ultradifferentiable classes are defined by estimates on
its derivatives (``generalized Cauchy estimates'') or by estimates on the Fourier transform determined by data like weight sequences, see e.g.~\cite{Komatsu} or weight functions, cf.~\cite{BMT90}.
In order to formulate and prove results we need to work with the data and impose conditions on
it. However, in part of the literature it appears that it would be enough in some situations that the ultradifferentiable classes
under consideration satisfy some basic axioms and there is no further need in the proofs to use the defining data of the classes.
In fact, in the papers \cite{BierstoneMilman} and \cite{Schapira68} for example
the authors are only working on the premise that the ultradifferentiable classes
satisfy certain invariance properties.

The present article tries to formulate an abstract theory
of ultradifferentiable classes and give some applications to the theory of partial differential equations or  theory of CR manifolds.
The choice of applications is based on some of the authors previous work,
e.g.~\cite{FuerdoesCR}, \cite{Fuerdoes20} or \cite{fuerdoes2025}.
Since the patterns of most of the proofs are already present in the literature we only indicate the arguments in some particular cases as explanatory motivations of the ideas
and concentrate on the defining axioms of the ultradifferentiable classes.
It is necessary to point out that it is not clear, if all axioms discussed in the 
different sections are really independent of each other, as is pointed out
in some cases. Moreover, since our focus is distinctly on geometric applications,
our approach excludes e.g.~anisotropic classes as for example used in 
\cite{Chinni23} and the classes introduced in \cite{LiessRodino84}.

The paper is organized in the following way: In Section \ref{Section2} we
introduce the fundamental concepts of this paper and discuss quasianalyticity of
the classes and applications to the regularity theory of analytic partial differential operators, based on the  approach in \cite{fuerdoes2025}.
In particular, we consider the case if there is a wavefront set associated
the classes
In Section \ref{Section3} we discuss the conditions necessary to
define ultradifferentiable manifolds and develop a differential geometric 
theory mirroring the smooth category. As an application
we prove a quasianalytic Holmgren theorem following \cite{Fuerdoes20}.
In Section \ref{Section4} we present the applications to CR theory.
In particular, we generalize the results of \cite{LamelMirRegSections}.
Finally, in Section \ref{Section5} we give an example of a family of ultradifferentiable classes satisfying our axioms, namely the classes
first defined in \cite{RainerSchindl14}.

\subsection*{Acknowledgments}
This work was funded in whole or in part by the Austrian Science Fund (FWF) 10.55776/PAT1994924.

The author would also like thank Bernhard Lamel for his interest in this work.
\section{Sheaves of smooth functions in PDE theory}\label{Section2}
\subsection{Isotropic sheaves}
We denote by $\E_n$ be the sheaf of smooth functions on $\R[n]$, $n\in\N$, i.e.
$\E_n(U)$ is the algebra of smooth functions on $U$ for any open set $U\subseteq\R[n]$.
Obviously, the restriction of any function $f\in\E_n(U)$ to any open set $U^\prime\subseteq U$
is an element of $\E_n(U^\prime)$. However, the restriction of $f$ to a set $V\subseteq U$ which
is open in $\R[m]$ ($m\leq n$), is also an element of $\E_m(V)$. We want to formalize this condition:
If $V\subseteq\R[m]$, $U\subseteq\R[n]$ are open sets with $m\leq n$ and $V\subseteq U$
then the inclusion mapping 
\begin{equation*}
\chi_{UV}:\; V\longrightarrow U
\end{equation*}
induces a mapping by pullback
\begin{equation*}
\lambda_{UV}=\chi^\ast_{UV}: \; \E(U)\longrightarrow\E(V)
\end{equation*}
given by $\lambda_{UV}(f)=f\circ \chi_{UV}=f\vert_V$ for $f\in\E(U)$
%\footnote{We identify here without loss of generality $U$}.

We want to consider subsheaves of $\E_n$ that inherite this property:
Let $\sheaf{}$ be a subsheaf of $\E$, i.e.~$\sheaf{}=(\sheaf{n})_{n\in\N}$ and $\sheaf{n}$ is a subsheaf of
$\E_n$ for each $n\in\N$. If $\analytic_n$ is the sheaf of real-analytic functions on $\R[n]$
then $\analytic=(\analytic_n)_{n\in\N}$ is a subsheaf of $\E$.
Finally we denote by $\poln$ the algebra of (complex) polynomials on $\R[n]$.
\begin{Def}
We say that $\sheaf{}$ is an isotropic ultradifferentiable
sheaf  if the following conditions are satisfied.
\begin{enumerate}[label=(U\arabic*),itemsep=0.3em,topsep=0.4em]
\item $\sheaf{}$ is an subsheaf of $\E$, i.e.~$\sheaf{n}$ is an subsheaf of $\E_n$ for any $n\in\N$. Furthermore, for each $n\in\N$ we have that $\poln\subseteq\sheaf[{\R[n]}]{n}$.
%\item If $\mathcal{P}_n$ is the vector space of polynomial functions on $\R^n$ then
%$\mathcal{P}_n\subseteq\E_n(\R[n])$ for any $n\in\N$.
\item\label{LinearInv} $\sheaf{}$ is closed under translations and dilations: 
Let $U\subseteq\R[n]$, $a\in\R[n]$ and $\lambda>0$.
If $\Psi(x)= \lambda^{-1} x-a$ then
\begin{equation*}
   f\in\sheaf[U]{n}\Longrightarrow \Psi^\ast f=f\circ\Psi\in\sheaf[\lambda U+a]{n}.
\end{equation*}
where $\lambda U+a=\Set{\lambda y+a\given y\in U}$.
\item Let %$\sheaf{}$ be an isotropic ultradifferentiable class and 
$U\subseteq\R[n]$, $V\subseteq\R[m]$ 
be two open sets with $V\subseteq U$ (i.e.~$m\leq n$).
Then 
\begin{equation*}
    \lambda_{UV}\left(\sheaf[U]{n}\right)\subseteq\sheaf[V]{m}.
\end{equation*}
%\end{Lem}
%\item Let $U\subseteq\R[n]$, $V\subseteq\R[m]$ be two open sets with $V\subseteq U$ (i.e.~$m\leq n$).
%Then 
%\begin{equation*}
%    \lambda_{UV}\left(\sheaf[U]{n}\right)\subseteq\sheaf[V]{m}.
%\end{equation*}
\item\label{Tensor} If $f\in\sheaf[U]{n}$ and $g\in\sheaf[V]{m}$ then $f\otimes g\in\sheaf[U\times V]{n+m}$.
\item $\sheaf{}$ is closed under involution: If $\varphi\in\sheaf[U]{n}$ then $\bar{\varphi}\in\sheaf[U]{n}$.
\end{enumerate}
\end{Def}

%\begin{Lem}
%Let $\sheaf{}$ be an isotropic ultradifferentiable class and $U\subseteq\R[n]$, $V\subseteq\R[m]$ 
%be two open sets with $V\subseteq U$ (i.e.~$m\leq n$).
%Then 
%\begin{equation*}
%    \lambda_{UV}\left(\sheaf[U]{n}\right)\subseteq\sheaf[V]{m}.
%\end{equation*}
%\end{Lem}

%\begin{proof}
%    Let $\iota: \R[m]\rightarrow \R[n]$ be the linear inclusion mapping such that 
 %   $\iota(V)\subseteq U$. 
 %   Then $\lambda_{UV}(f)= (\iota^\ast f)\vert_V$ for any $f\in\E_n(U)$.
 %   If $f\in \sheaf[U]{n}$ then $(\iota^\ast f\in \sheaf[\iota^{-1}(U)]{m}$ by \ref{LinearInv}.
  %  On the other hand $g\vert_V\in\sheaf[V]{m}$ for every $g\in\sheaf[\iota^{-1}(U)]{m}$. 
%\end{proof}

\begin{Def}[Quasianalyticity]
    Let $U\subseteq\R[n]$ be an open set and $\D(U)$ the algebra of smooth functions with compact support in $U$. Recall that a subalgebra $\mathcal{A}\subseteq\E(U)$ is called quasianalytic
    if $\mathcal{A}\cap\D(U)=\{0\}$ and non-quasianalytic otherwise.
\end{Def}
\begin{Lem}\label{QuasiLemma1}
    Let $\sheaf{}$ be an isotropic ultradifferentiable class. For each $n\in\N$ the following statements are equivalent:
    \begin{enumerate}
        \item There is some open set $U\subseteq \R[n]$ the algebra $\sheaf[U]{n}$ is non-quasianalytic.
        \item $\sheaf[U]{n}$ is non-quasianalytic for each open $U\subseteq\R[n]$.
    \end{enumerate}
\end{Lem}
\begin{proof}
    One direction is trivial. Suppose now that there is some $U_0$ such that
    $\sheaf[U_0]{n}\cap\D(U_0)\neq\{0\}$. It follows directly that $\sheaf[U]{n}$ is non-quasianalytic
    for all open $U\subseteq\R[n]$ with $U_0\subseteq U$.

    Let now $f\in\sheaf[U_0]{n}\cap\D(U_0)$ with $f(x_0)\neq 0$ for some $x_0\in U_0$.
    Since $\supp f$ is compact we may assume without loss of generality that $U_0$ is bounded.
    For any affine map $\Psi$ we have that $\Psi^\ast f=f\circ\Psi\in\sheaf[\Psi^{-1} (U_0)]{n}\cap 
    \D(\Psi^{-1}(U_0))$ by \ref{LinearInv}. Now for each $U\subseteq\R^n$ there is a affine map $\Psi$ such
    that $\Psi^{-1}(U_0)\subseteq U$. Thence $\sheaf[U]{n}\cap\D(U)\neq\{0\}$.
\end{proof}
\begin{Rem}\label{QuasiRemark0}
    The proof of Lemma \ref{QuasiLemma1} implies the existence os special functions
    in non-quasianalytic spaces $\sheaf[n]{U}$.
    Indeed, if $\sheaf[U]{n}$ is non-quasianalytic then there is some function
    $g\in\sheaf[U]{n}\cap\D(U)$ and some point $z_0\in U$ such that 
    $\zeta=g(z_0)\neq 0$. 
    Let $W\subsetneq U$ be an bounded open neighborhood of $\supp g$.
    Moreover, let $p\in U$ be an arbitrary point of $U$ and $V\subsetneq U$
    a neighborhood of $p$ in $U$. Then there is an affine
    map $\Psi:\,\R[n]\rightarrow\R[n]$  such that $W\subseteq\Psi(V)$.
    It follows again from \ref{LinearInv} that $\Phi^{\ast}(g)\in\sheaf[V]{n}$
\end{Rem}
We may say that the sheaf $\sheaf{n}$ is non-quasianalytic if one of the conditions of 
Lemma \ref{QuasiLemma1} is satisfied and quasianalytic otherwise.
Then $\sheaf{n}$ is quasianalytic if and only if $\sheaf[U]{n}$ is quasianalytic for any open set
$U\subseteq\R[n]$.
\begin{Lem}\label{QuasiLemma2}
Let $\sheaf{}$ be an isotropic ultradifferentiable class and $n\in\N$. Then $\sheaf{1}$ is non-quasianalytic if and only if $\sheaf{n}$ is non-quasianalytic.
\end{Lem}
\begin{proof}
Following Lemma \ref{QuasiLemma1} it is enough to find two open sets
$U\subseteq\R[n]$ and $V\subseteq\R[]$ such that
$\sheaf[U]{n}$ is non-quasianalytic if and only if $\sheaf[V]{1}$ is non-quasianalytic.
We choose $V=(-1,1)$ and $U=V^n=(-1,1)^n$.
First note that 
to each $h\in\sheaf[V]{1}\cap\D(V)$ we can associate the function
\begin{equation*}
    \tilde{h}(x)=\prod_{j=1}^n h(x_j)\qquad x=(x_1,\dotsc,x_n)\in U=V^n
\end{equation*}
and $\tilde{h}\in\sheaf[U]{n}\cap \D(U)$ by \ref{Tensor}. 
Moreover, $\tilde{h}(0)=(h(0))^n$. It follows that the non-quasianalyticity of $\sheaf[V]{1}$
implies that of $\sheaf[U]{n}$, cf. Remark \ref{QuasiRemark0}.

For the other direction we need more preparation.
If $\rho:\,\R\rightarrow\R[n]$ is the linear map given by $t\mapsto(t,0,\dotsc,0)$
then $\rho^{-1}(U)=V$ and $\rho\vert_V=\chi_{UV}$.
If $\sheaf[U]{n}$ is non-quasianalytic then by Remark \ref{QuasiRemark0} there exists
 a function $g\in\sheaf(U){n}\cap\D(U)$ with $g(0)=1$. 
 Then $f=\lambda_{UV}(g)=g\circ\rho\in\sheaf[V]{1}\cap\D(V)$ and $f(0)=g(0)=1$. 
Thus $\sheaf[V]{1}$ is non-quasianalytic.

\end{proof}
\begin{Rem}
Combining  both Lemmata \ref{QuasiLemma1} and \ref{QuasiLemma2} above we see that $\sheaf[U]{n}$
is non-quasianalytic for every open set $U\subseteq\R[n]$ and each $n\in\N$ if and only
if $\sheaf[V]{m}$ is non-quasianalytic for some open set $V\subseteq\R[m]$.
In that case we say that the isotropic ultradifferentiable sheaf $\sheaf{}$ is non-quasianalytic
and quasianalytic otherwise.
\end{Rem}

%\subsection{(Non-)Quasianalyticity}

\subsection{Microlocal sheaves}
\begin{Def}
    We say that an isotropic ultradifferentiable sheaf $\sheaf{}$ is semiregular if 
    the following conditions hold:
    \begin{enumerate}[label=(S\arabic*),itemsep=0.3em,topsep=0.2em]
        \item\label{AnalyticIncl} $\analytic$ is a subsheaf of $\sheaf{}$.
        \item\label{DerivClosed} $\sheaf{}$ is closed under derivation: For all $U\subseteq\R[n]$ open the following holds
\begin{equation*}
    f\in\sheaf[U]{n}\Longrightarrow D_jf\in\sheaf[U]{n},\;j=1,\dotsc,n.
\end{equation*}
\item\label{DivCoordinate} $\sheaf{}$ is closed under division by a coordinate: Let $f\in\sheaf[U]{n}$
such that $f(x_1,\dotsc,x_{j-1},a,x_{j+1},\dotsc,x_n)=0$ for some $a$ and $j\in\{1,\dotsc,n\}$
then there exists $g\in\sheaf[U]{n}$ such that
\begin{equation*}
    f(x)=\bigl(x_j-a\bigr)g(x),\qquad x\in U.
\end{equation*}
\item\label{Comp1} $\sheaf{}$ is closed under analytic mappings: Let $U\subseteq\R[n]$ and $V\subseteq\R[m]$ be open
sets and $\Phi: U\rightarrow V$ be a real-analytic map.
Then
\begin{equation*}
    f\in\sheaf[V]{m}\Longrightarrow \Phi^\ast f=f\circ\Phi\in\sheaf[U]{n}.
\end{equation*}
    \end{enumerate}
\end{Def}

\begin{Rem}\label{SemiRem}\leavevmode
    \begin{itemize}
\item Iterating \ref{DerivClosed} we see that for any $\alpha\in\N_0^n$ we have that
$D^\alpha f\in\sheaf[U]{n}$ for all $f\in\sheaf[U]{n}$. 
In fact, since $\sheaf[U]{n}$ is an algebra we have that 
\begin{equation}\label{PClosed}
    P\left(\sheaf[U]{n}\right)\subseteq\sheaf[U]{n}
\end{equation}
for any linear differential operator $P$ of the form
\begin{equation*}
    P(x,D)=\sum_{\abs{\alpha}\leq d} p_\alpha(x)D^\alpha,\qquad p_\alpha\in\sheaf[U]{n}.
\end{equation*}
Using also \ref{AnalyticIncl} we see that $\sheaf{n}$ is closed under analytic differential operators with real-analytic differential operators.
\item Property \ref{Comp1} allows us to define the sheaf $\sheaf{M}$ of functions of class $\sheaf{}$
on real-analytic manifolds $M$, cf.~the next section and \cite[Section 8]{HoeBook1}
\item 
We may note that it does not appear to be clear if \ref{DerivClosed} or $\ref{DivCoordinate}$ are independent of each other in view of
the Fundamental Theorem of Calculus. 
In concrete examples of ultradifferentiable sheaves like Denjoy-Carleman classes,
one can use the Fundamental Theorem of Calculus to deduce \ref{DivCoordinate}
from \ref{DerivClosed}, but this involves the defining estimates of the Denjoy-Carleman classes. It may be of interest to find an argument which works without
the use of the defining data.
    \end{itemize}

\end{Rem}

%\begin{Prop}
  %  Let $\sheaf{}$ be an isotropic ultradifferentiable sheaf.
  %  Then $\sheaf{}$ is closed under the division by a coordinate:
% Let $U\subseteq\R[n]$ (with coordinates $x=(x_1,\dotsc,x_n)$) and
% $j=1,\dotsc, n$. Suppose that $a_j$
%\end{Prop}
For each open set $U\subseteq\R[n]$ we denote by $S^\ast U=U\times S^{n-1}$ the co-spherical bundle of $U$. Note that there is a natural projection $\pi: S^\ast U\rightarrow U$ to the base.
We set
\begin{equation*}
\pot[U]{n}:=\Set*{\mathcal{V}\subseteq S^\ast U\given \mathcal{V} \text{ is a closed subset of } S^\ast U}.
\end{equation*}
We claim that $\pot{n}:U\mapsto \pot[U]{n}$, together with the restriction maps
\begin{equation*}
    \rho_{UV}: \pot[U]{n}\ni\crb\longmapsto \crb\cap S^\ast V\in\pot[V]{n} 
\end{equation*}
for $V\subseteq U\subseteq\R[n]$ defines a sheaf of sets on $\R[n]$: 
Obviously $\rho_{UU}=\Id_{\pot[U]{n}}$ for any open set $U\subseteq\R[n]$.
Moreover, for open sets $W\subseteq V\subseteq U\subseteq\R[n]$ we have that
\begin{equation*}
    \rho_{VW}\left(\rho_{UV}(\crb)\right)=\left(\crb\cap S^\ast V\right)\cap S^\ast W=\crb\cap S^\ast W
    =\rho_{UW}(\crb)
\end{equation*}
for all $\crb\in\pot[U]{n}$.

Finally, let $U\subseteq \R[n]$ be open and $(U_j)_{j\in J}$ be an open covering of $U$,
i.e.~the sets $U_j\subseteq U$ are open and $\bigcup U_j=U$.
Suppose we are given $\crb_j\in\pot[U_j]{n}$ for each $j\in J$ such that $\crb_j\cap S^\ast U_k=\crb_k\cap S^\ast U_j$ for each
pair $j,k\in J$ with $U_k\cap U_j\neq\emptyset$.
Then there is trivially a unique set $\crb$ such that 
\begin{equation*}
    \rho_{UU_j}(\crb)=\crb\cap S^\ast U_j=\crb_j\qquad \fa j\in J.
\end{equation*}
We claim that $\crb$ is closed. If $\crb$ would be not closed in $S^\ast U$ then the set 
$\mathcal{U}=S^\ast U\setminus\crb$ is not open, i.e.~there is a point $q\in\mathcal{U}$
such that for any neighborhood $\mathcal{W}$ of $q$ we have that
$\mathcal{W}\cap\crb\neq\emptyset$. 
If $j$ is such that $\pi(q)\in U_j$ then it follows that
$\mathcal{W}\cap S^\ast U_j\cap\crb=\mathcal{W}\cap \crb_j\neq\emptyset$, i.e.~every
neighborhood of $q$ in $S^\ast U_j$ has non-empty intersection with $\crb_j$.
Since $\crb_j$ is closed it follows that $q\in\crb_j\subseteq\crb$.

Analogously to before
we set $\pot{}=(\pot{n})_{n\in\N}$. Moreover let $\Distr{n}$ be the sheaf of distributions on $\R[n]$
and $\Distr{}=(\Distr{n})_{n\in\N}$.

\begin{Def}
    An semiregular ultradifferentiable sheaf $\sheaf{}$ is microlocal if there exists a 
    sheaf morphism
    $\WFR: \Distr{}\rightarrow\pot{}$, i.e.~for each $n\in\N$ and every open set $U\subseteq\R[n]$
    there is a mapping
    \begin{equation*}
\WFR: \Distr[U]{n}\ni u\longmapsto \WFR u\in\pot[U]{n},
    \end{equation*}
    with the following properties:
    \begin{enumerate}[label=(M\arabic*),itemsep=0.3em,start=0]
    \item\label{BasicWF0} $\WFR (u_1+u_2)\subseteq\WFR u_1+\WFR u_2$ for all $u_1,u_2\in\Distr[U]{n}$.
        \item\label{WFComparision} For all $n\in\N$, $U\subseteq\R[n]$ open and $u\in\Distr[U]{n}$ we have that
        $\WF u\subseteq\WFR u\subseteq\WF_A u$, where $\WF u$ is the (smooth) wavefront set and 
        $\WF_A u$ is the analytic wavefront set of $u$.
        %\item For each $n\in\N$ and $V\subseteq U\subseteq\R[n]$ open we have that
        %\begin{equation*}
         %   \WFR \left(u\vert_V\right)=\WFR u\cap S^\ast V
        %\end{equation*}
        %for every $u\in\Distr[U]{n}$.
        \item\label{WFSing} $\pi(\WFR u)=\singsupp_{\sheaf{}}u$ for all $u\in\Distr[U]{n}$
        where the $\sheaf{}$-singular support of $u$, $\singsupp_{\sheaf{}}u$,
        is the the set of points $p\in U$ such that 
        $u\vert_V\notin\sheaf[V]{n}$ for all neighborhoods $V$ of $p$.
        \item\label{WFInvolution} For each $u\in\Distr[U]{n}$ we have that
        \begin{equation*}
            \WFR \bar{u}=\sigma \left(\WFR u\right)
        \end{equation*}
        where $\sigma: S^\ast(U)\rightarrow S^\ast U$ is the antipodal mapping given by
        $\sigma(p,\xi)=(p,-\xi)$.
        %\item The microlocal elliptic regularity theorem holds for analytic operators: Let $U\subseteq\R[n]$ be open and
       %$P=\sum_{\abs{\alpha}\leq d} p_\alpha D^\alpha$ be an analytic differential operator 
        %in $U$ (i.e.~$p_\alpha\in\analytic(U)$). Then
        %\begin{equation*}
         %   \WFR Pu\subseteq\WFR u\subseteq\WFR Pu\cup\Char P
        %\end{equation*}
        %for every $u\in\Distr[U]{n}$. Here 
      %  \begin{equation*}
        %    \Char P=\Set*{(x,\xi)\in S^\ast U\given p_d(x,\xi)=0}
       % \end{equation*}
       % is the characteristic set of $P$ and 
        %\begin{equation*}
           % p_d(x,\xi)=\sum_{\abs{\alpha}=d}p_\alpha(x)\xi^\alpha
        %\end{equation*}
        %is the principal symbol of $P$.
        \item\label{WFanalMap} Let  $U\subseteq\R[n]$, $V\subseteq\R[m]$ be open sets
        and $\Phi: U\rightarrow V$ be an analytic mapping and let
        \begin{equation*}
            N_\Phi=\Set*{(\Phi(x),\eta)\in S^{\ast}V\given \prescript{t}{}{\Phi}^\prime(x)\eta=0}
        \end{equation*}
        be the set of normals of $\Phi$. Then the pull-back $\Phi^\ast u\in\Distr[U]{n}$ is well-defined and
        \begin{equation*}
            \WFR\left(\Phi^\ast u\right)\subseteq\Phi^\ast\left(\WFR u\right)=
            \Set*{\left(x,\frac{\prescript{t}{}{\Phi}^\prime(x)\eta}{\abs*{\prescript{t}{}{\Phi}^\prime(x)\eta}}\right)
            \in S^\ast U\given (\Phi(x),\eta)\in S^\ast V}
        \end{equation*}
        for $u\in\Distr[V]{m}$ with $\WFR u\cap N_\Phi=\emptyset$.
        \item\label{WFproj} Let $u\in\mathcal{E}^\prime(\R[n+m])$ and $\sheaf{}$ be a microlocal sheaf. If $u_1\in\mathcal{E}^\prime(\R[n])$ is the distribution given by
        \begin{equation*}
            \langle u_1,\varphi\rangle=\langle u, \varphi\otimes 1\rangle,\qquad \varphi\in\E(\R^n)
        \end{equation*}
        then
        \begin{equation*}
            \WFR u_1\subseteq\Set*{(x,\xi)\in S^\ast \R[n]\given \ex y\in\R[m]\;\,(x,y,\xi,0)\in \WFR u}.
        \end{equation*}
        \item\label{WFTensor} Let $U\subseteq\R[n]$, $V\subseteq\R[m]$ be two open sets and $u\in\Distr[U]{n}$, $v\in\Distr[V]{m}$. Then $u\otimes v\in\Distr[U\times V]{n+m}$ and 
        \begin{equation*}
            \WFR (u\otimes v)\subseteq \left(\WFR u\times\WFR v\right) \cup\left((\supp u\times\{0\}\times\WFR v)\right)\cup
            \left(\WFR u\times(\supp v\times\{0\})\right).
        \end{equation*}
    \end{enumerate}
\end{Def}
A close inspection of \cite[Section 8.5]{HoeBook1} yields that we can prove
the following three statements for microlocal sheaves, cf.~\cite[Section 2]{fuerdoes2025}. 
\begin{Prop}\label{WFProp1}
    Let $\sheaf{}$ be a microlocal sheaf and $u,v\in\Distr[U]{n}$ such that
    $\WFR u\cap\sigma(\WFR v)=\emptyset$. Then $uv\in\Distr[U]{n}$  is well-defined
    and 
    \begin{equation*}
        \WFR (uv)\subseteq\Set*{\left(x,\frac{\xi+\eta}{\abs{\xi+\eta}}\right)\in
        S^\ast U\given (x,\xi)\in\WFR u\text{ or } \xi=0,\; (x,\eta)\in\WFR v
        \text{ or }\eta=0}.
    \end{equation*}
\end{Prop}
\begin{proof}
    As in \cite[Theorem 8.2.10]{HoeBook1} the product is defined as
    $uv\coloneqq \delta^\ast(u\otimes v)$ where
    $\delta: U\rightarrow U\times U$ is the diagonal mapping $u(x)=(x,x)$.
    Then $\prescript{t}{}{\Phi}^\prime(x)(\xi,\eta)=\xi+\eta$ and thus
    the statement follows from \ref{WFanalMap} and \ref{WFTensor}.
\end{proof}

\begin{Cor}\label{WFCor1}
Let $\sheaf{}$ be a microlocal sheaf and $a\in\sheaf[U]{n}$.
Then
\begin{equation*}
    \WFR (au)\subseteq\WFR u \qquad \fa u\in\Distr[U]{n}.
\end{equation*}
\end{Cor}
\begin{proof}
    The product of two distributions defined above extends the usual product
    of distributions with smooth functions, cf.~\cite[p.~267]{HoeBook1}.
    Thence Proposition \ref{WFProp1} implies that 
    \begin{equation*}
        \WFR (au)\subseteq\Set*{(x,\eta)\in S^\ast U\given (x,\eta)\in \WFR u}.
    \end{equation*}
\end{proof}

\begin{Prop}\label{WFProp2}
    Let $\sheaf{}$ be a microlocal sheaf,  $U\subseteq\R[n]$ and $V\subseteq\R[m]$ be open sets and
        $K\in\Distr[U\times V]{n+m}$ be a distribution such that the map $\supp K\rightarrow U$ is proper. Then
        \begin{equation*}
            \WFR \mathsf{K}u\subseteq\Set*{(x,\xi)\in S^\ast U\given \ex y\in\supp u:\;
            (x,y;\xi,0)\in\WFR K},\qquad u\in\sheaf[V]{m}
        \end{equation*}
        where $\mathsf{K}$ is the linear operator with kernel $K$.
\end{Prop}
\begin{proof}
We can use the proof of \cite[Theorem $8.5.4^\prime$]{HoeBook1} nearly verbatim:
After replacing $K$ with $(1\otimes u)K$ we can assume that $\varphi\vert_W=1$.
Let now $W\Subset U\times V$ and $\varphi\in\D(U\times V)$ such that $u\vert_W\equiv 1$.
Then applying \ref{WFproj} gives the statement for $\varphi K$ and since the choice
of $W$ has been arbitrary the general statement follows.
\end{proof}
For the last statement we need some more notations:
For a microlocal sheaf $\sheaf{}$, open sets $U\subseteq\R[n]$, $V\subseteq\R[m]$
and $u\in\Distr[U\times V]{n+m}$ we set
\begin{align*}
    \WF_{\sheaf{}}(u)_U&\coloneqq\Set*{(x,\xi)\in S^\ast U\given \ex y\in V\;\; (x,y,\xi,0)\in\WFR u}\\
    \WF^\prime_{\sheaf{}}(u)&\coloneqq\Set*{(x,y,\xi,\eta)\in S^{\ast}(U\times V)\given (x,y,\xi,-\eta)\in\WFR u}\\
    \WF^\prime_{\sheaf{}}(u)_V&\coloneqq \Set*{(y,\eta)\in S^\ast V\given \ex x\in U\;\; (x,y,0,\eta)\in\WFR u}
\end{align*}
\begin{Prop}\label{WFProp3}
    Let $\sheaf{}$ be a microlocal sheaf, $U\subseteq\R[n]$, $V\subseteq\R[m]$ be open sets, $K\in\Distr[U\times V]{n+m}$.  
If $u\in\E^\prime(V)$ such that $\WFR u\cap \WF^\prime_{\sheaf{}}(K)_V$
then
    	\begin{equation*}
		\begin{split}
			\WFR(\mathsf{K}u)&\subseteq\WF_{\sheaf{}}(K)_U\cup\left(
			\WF_{\sheaf{}}^\prime(K)\circ\WFR u\right)\\
			&=\WF_{\sheaf{}}(K)_U\cup\Set*{(x,\xi)\in S^\ast U\given
            \ex (x,\eta)\in\WFR u:\;\;
				\left(x,y,\frac{\xi}{\abs{\xi}+\abs{\eta}},-
                \frac{\eta}{\abs{\xi}+\abs{\eta}}\right)\in\WFR(K)}.
		\end{split}
	\end{equation*}
\end{Prop}
\begin{proof}
    We can again follow the proof of \cite[Theorem 8.5.5]{HoeBook1}, cf.~\cite{fuerdoes2025}, with only very minor modifications.
    It follows from \ref{WFTensor} that
    \begin{equation*}
        \WFR(1\otimes u)\subseteq\Set*{(x,y,0,\eta)\in S^\ast(U\times V\given
        (y,\eta)\in\WFR u}.
    \end{equation*}
    If we set $K_u=K(1\otimes u)$ then Proposition \ref{WFProp1} gives that
    \begin{equation*}
    \begin{split}
        \WFR K_u&=\Set*{\left(x,y,\xi,\frac{\eta+\eta^\prime}{\abs{\eta+\eta^\prime}}\right)\in S^\ast(U\times V)\given (y,\eta)\in\WFR u,\;(x,y,\xi,\eta^\prime)\in\WFR K}\\
        &\qquad \cup \WFR K\cup\WFR(1\otimes u).
    \end{split}
    \end{equation*}
    Since $\mathsf{K}u=\langle  K_u,\,.\,\otimes 1\rangle$
    the assertion follows from \ref{WFproj}.
\end{proof}
\subsection{Ultradifferentiable hypoellipticity of analytic partial differential operators}
\begin{Def}
Let $u\subseteq\R[n]$ be an open set and $P$ be a linear differential operator with coefficients in $\sheaf[U]{n}$. If $\sheaf{}$ is a semiregular sheaf then we say that
$P$ is $\sheaf{}$-hypoelliptic in $U$ if 
\begin{equation*}
    \singsupp_{\sheaf{}}Pu=\singsupp_{\sheaf{}}u
\end{equation*}
for all $u\in\Distr[U]{n}$.
\end{Def}

As a first application we show that $\WFR$ is invariant under the action of linear
partial differential operators with coefficients in $\sheaf{}$.
More precisely 
we notice we can generalize the main statements of
\cite{fuerdoes2025} to microlocal sheaves $\sheaf{}$.
First, combining \cite[Theorem 3.1]{fuerdoes2025} with
Proposition \ref{WFProp3} 
gives the following statement of the $\sheaf{}$-microlocal properties
of analytic Fourier integral operators. For the notions used, see 
\cite{TrevesBook80a,TrevesBook80b,TrevesBook25}.

\begin{Thm}
    Let $U\subseteq\R[n]$ and $V\subseteq\R[n]$ be two open sets and $\Gamma\subseteq\R[N]$ an open cone.
    Consider $A:\E_m^\prime(V)\rightarrow \Distr[U]{n}$ be an analytic Fourier integral operator with (real-valued) phase-function $\Phi$ on $U\times V\times\Gamma$.
Then
	\begin{equation}
		\WFR Au\subseteq\mathfrak{R}_\Phi(\WFR u),
		\qquad \fa u\in \E^\prime(\Omega_2)
	\end{equation}
	for any microlocal sheaf $\sheaf{}$, where
	\begin{equation*}
		\mathfrak{R}_\Phi(E)=\Set*{(x,\xi)\in T^\ast\Omega_1\setminus\{0\}\given \ex (y,\eta)\in E\;
			\text{with } \xi=d_x\Phi(x,y,\theta)\;\&\; \eta=-d_y\Phi(x,y,\theta)\text{ for some $\theta\in\Gamma$}
		}.
	\end{equation*}
	for any subset $E\subseteq T^\ast\Omega_2\!\setminus\{0\}$.
\end{Thm}
Another case are analytic pseudodifferential operators, cf.~\cite{TrevesBook25}.
\begin{Thm}\label{AnalyticPseudo}
Let $P$ be an analytic pseudodifferential operator defined on some open set $U\subseteq\R[n]$ and $\sheaf{}$ be a microlocal sheaf. Then
\begin{equation*}
    \WFR Pu\subseteq\WFR u\qquad \fa u\in\E^\prime(U).
\end{equation*}
\end{Thm}

\begin{Cor}
If  $\sheaf{}$ is a microlocal sheaf then the following holds:
\begin{enumerate}[label=\emph{(M\arabic*$^\prime$)},start=5]
    \item\label{WFsmoothing} $P$ be a differential operator with coefficients in $\sheaf[U]{n}$.
    Then
    \begin{equation*}
        \WFR Pu\subseteq\WFR u\qquad \qquad \fa u\in\Distr[U]{n}.
    \end{equation*}
\end{enumerate}
\end{Cor}
\begin{proof}
First, for any $\alpha\in\N_0$ we have that
\begin{equation*}
    \WFR D^\alpha u\subseteq\WFR u\qquad u\in\E^\prime(U)
\end{equation*}
by Theorem \ref{AnalyticPseudo}. If $v\in\Distr[U]{n}$, $x\in U$ and
$\varphi\in\D(U)$ such that $\varphi\equiv 1$ near $x$.
Then $D^\alpha (\varphi v)= \varphi D^\alpha v+ \psi_\alpha$ where
$\psi_\alpha$ is a distribution which vanishes in some neighborhood of $x$.
Thus we have that
\begin{equation*}
    \WFR D^\alpha v\subseteq \WFR v.
\end{equation*}
Combining this with Corollary \ref{WFCor1} gives the assertion.
\end{proof}
Since any elliptic analytic pseudodifferential operator has an analytic
parametrix in the same class, see e.g.~\cite{TrevesBook25}, we can directly infer the following result.
\begin{Thm}
Let $\sheaf{}$ be a microlocal sheaf and $A$ be an elliptic analytic pseudodifferential operator in $U$.
Then $\WFR Au=\WFR u$ for all $u\in\E^\prime_n(U)$.
\end{Thm}
Thus in the case of differential operators we obtained the ultradifferentiable microlocal elliptic theorem for analytic differential operators.
For a given differential operator $P=\sum_{\abs{\alpha}\leq d}p_\alpha D^\alpha$
with principal symbol $p_d(x,\xi)=\sum_{\abs{\alpha}=d} p_\alpha \xi^\alpha$
we define the characteristic set of $P$ as
\begin{equation*}
    \Char P=\Set*{(x,\xi)\in S^\ast U\given p_d(x,\xi)=0}.
\end{equation*}
\begin{Thm}\label{MicroEllThm1}
    Let $\sheaf{}$ be a microlocal sheaf. Then we have that:
    \begin{enumerate}[label=\emph{(M\arabic*$^\prime$)},start=6]
        \item\label{AnalEllWF} If $P$ be a differential operator with coefficients in $\analytic_n(U)$
        then
        \begin{equation*}
        \WFR Pu\subseteq\WFR u\subseteq\WFR Pu\cup\Char P \qquad\fa u\in\Distr[U]{n}.
    \end{equation*}
    where $\Char P$ is the characteristic set of $P$.
    \end{enumerate}
\end{Thm}
\begin{Cor}
    Let $\sheaf{}$ be a microlocal sheaf and $P$ be an elliptic differential operator
    with coefficients in $\analytic_n(U)$.
    Then $P$ is $\sheaf{}$-hypoelliptic.
\end{Cor}
We can also follow the arguments in \cite[Section 4]{fuerdoes2025} to generalize
the famous characterization of hypoellipticity of analytic differential operators of principal type
by Treves \cite{Treves71}, cf.~also \cite{TrevesBook25}.
\begin{Thm}
    Let $P$ be an analytic differential operator of principal type and $\sheaf{}$
    be a microlocal sheaf. Then the following statements are equivalent.
    \begin{enumerate}[label=(\roman*),itemsep=0.1em]
        \item $P$ is (smooth) hypoelliptic.
        \item $P$ is $\sheaf{}$-hypoelliptic.
        \item $P$ is analytic hypoelliptic.
    \end{enumerate}
\end{Thm}

\begin{Thm}\label{WFConv}
    Let $\sheaf{}$ be a microlocal sheaf, $k\in\Distr[{\R[n]}]{n}$ and 
    $u\in\E^\prime_n(\R[n])$. Then
    \begin{equation*}
        \WFR (k\ast u)\subseteq\Set*{(x+y,\xi)\in S^\ast\R[n]\given (x,\xi)\in \WFR k\text{ and }(y,\xi)\in\WFR u}.
    \end{equation*}
\end{Thm}
\begin{proof}
By definition of the convolution of distributions the map
\begin{equation*}
    \E_n^\prime(\R[n])\ni u\longmapsto k\ast u\in\Distr[{\R[n]}]{n}
\end{equation*}
possesses the kernel $K=\rho^\ast k\in\Distr[{\R[2n]}]{2n}$ where
$\rho: \R[n]\times\R[n]\rightarrow \R[n]$ is given by $\rho(x,y)=x-y$.  Then
\begin{equation*}
    \prescript{t}{}{\rho^\prime}(x,y) \xi =(\xi,-\xi)
\end{equation*}
and \ref{WFanalMap} implies that
\begin{equation*}
    \WFR K\subseteq\Set*{\left(x,y,\frac{\xi}{\sqrt{2}},-\frac{\xi}{\sqrt{2}}\right)\in S^{\ast}\R[2n]\given (x-y,\xi)\in\WFR k}
\end{equation*}
On the other hand, for any constant $c\in\R[n]$ the map $f_c(x)=(x+c,c)$ is a right
inverse of $\rho$ and $\prescript{t}{}{f_c^\prime}(x)(\xi,\eta)=\xi$ 
and thus $\WFR K\cap N_{f_c}=\emptyset$. Thence $f_c^\ast K=f_c^\ast (\rho^\ast k)= (\rho\circ f_c)^\ast k= k$ and by \ref{WFanalMap}
\begin{equation*}
    \WFR k\subseteq \Set*{(x,\xi)\in S^\ast \R[n]\given \left(x+c,c,\frac{\xi}{\sqrt{2}},-\frac{\xi}{\sqrt{2}}\right)\in\WFR K}.
\end{equation*}
It follows that
\begin{equation}
    \WFR K=\Set*{\left(x,y,\frac{\xi}{\sqrt{2}},-\frac{\xi}{\sqrt{2}}\right)\in S^{\ast}\R[2n]\given (x-y,\xi)\in\WFR k}.
\end{equation}
Applying Proposition \ref{WFProp3} finishes the proof.
\end{proof}

\begin{Cor}\label{WFCor3}
Let $\sheaf{}$ be a microlocal sheaf. Then the following statement holds:
\begin{enumerate}[label=\emph{(S\arabic*)},start=5]
    \item\label{Convolution} Let $k\in\Distr[{\R[n]}]{n}$ and $u\in\E_n^\prime(\R[n])$. Then
    \begin{equation*}
        \singsupp_{\sheaf{}} (k\ast u)\subseteq \singsupp_{\sheaf{}} k+\singsupp_{\sheaf{}}u.
    \end{equation*}
    In particular, if $k\in\sheaf[{\R[n]}]{n}$ then $k\ast u\in\sheaf[{\R[n]}]{n}$.
    On the other hand, if $u\in\sheaf[{\R[n]}]{n}\cap\E_n^\prime(\R[n])$ then
    $k\ast u\in\sheaf[{\R[n]}]{n}$ for any $k\in\Distr[{\R[n]}]{n}$,
\end{enumerate}
\end{Cor}
\begin{proof}
  The first statement is just a combination of Theorem \ref{WFConv} and 
  \ref{WFSing}.
  Moreover, if $k\in\sheaf[{\R[n]}]{n}$ then $\WFR K=\emptyset$ by
  the proof of Theorem \ref{WFConv}
  and therefore $k\ast u\in\sheaf[{\R[n]}]{n}$  for all $u\in\Distr[{\R[n]}]{n}$.
  On the other hand if $u\in\sheaf[{\R[n]}]{n}\cap \D(\R[n])$ then
  $\WFR (k\ast u)=\emptyset$ for all $k\in\Distr[{\R[n]}]{n}$ according
  to the proof of Theorem \ref{WFConv}.
\end{proof}

\begin{Thm}\label{ThmCPDO1}
    Let $\sheaf{}$ be a semiregular such that \ref{Convolution} holds and $P$ be
    a differential operator with constant coefficients.
    Then the following statements are equivalent:
    \begin{enumerate}[label=(\roman*),itemsep=0.1em]
        \item $P$ is $\sheaf{}$-hypoelliptic in any open set $U\subseteq\R[n]$.
        \item For every open set $U\subseteq\R[n]$ we have $\Set*{u\in\Distr[U]{n}\given Pu=0}\subseteq\sheaf[U]{n}$.
        \item All fundamental solutions $E$ of $P$ satisfy $\singsupp_{\sheaf{}}E=\{0\}$.
        \item There is a fundamental solution $E$ of $P$ such that 
        $\singsupp_{\sheaf{}} E=\{0\}$.
    \end{enumerate}
    If $\sheaf{}$ is actually a microlocal sheaf then any of the above statements
    are equivalent to
    \begin{enumerate}[label=(\roman*),resume]
        \item $P$ is $\sheaf{}$-microhypoelliptic in any open set $U\subseteq\R[n]$, i.e.~$\WFR Pu=\WFR u$
        for all $u\in\Distr[U]{n}$.
    \end{enumerate}
\end{Thm}
\begin{proof}
    It is clear that (i) implies (ii). Using (ii) for $U=\R[n]\setminus\{0\}$ gives (iii) which in turn trivially includes (iv). 
    In order to show that (iv) implies (i) let $x\in U$ be arbitrary and assume that
    $Pu$ is of class $\sheaf{}$ near $x$, i.e~there is a neighborhood $V$ of $x$
    such that $(Pu)\vert_V\in\sheaf[V]{n}$. Let $W\Subset V$ be another neighborhood of $x$ and $\psi\in\D(V)$ such that $\psi\vert_W\equiv 1$.
    Since $E$ is a fundamental solution of $P$ we conclude that
    \begin{equation*}
        \psi u=\delta\ast(\psi u)=PE\ast (\psi u)=E\ast P(\psi u).
    \end{equation*}
    Using \emph{\ref{Convolution}} we conclude that $x\notin\singsupp_{\sheaf{}} u$.

    If $\sheaf{}$ is a microlocal sheaf then $(v)\Rightarrow (i)$ by \ref{WFSing}.
    On the other hand, Theorem \ref{WFConv} shows that (iv) implies (v):
    If $(x,\xi)\notin \WFR Pu$ for $u\in\Distr[U]{n}$ and $\varphi\in\D(U)$ such that $\varphi= 1$ near $x$ then
    we have as before that $\varphi u=E\ast P(\varphi u)$.
    According to Theorem \ref{WFConv} $(x,\xi)\notin\WFR \varphi u$,
    which proves the assertion.
\end{proof}
We may refer to a differential operator with constant coefficients as
$\sheaf{}$-hypoelliptic if one of the conditions in Theorem \ref{ThmCPDO1} is
satisfied.
It is a a classical result that a differential operator with constant coefficients is
analytic hypoelliptic if and only if the operator is elliptic \cite{Petrowsky1939}
cf.~also \cite{RodinoBook}. 
The next result extends this theorem to quasianalytic semiregular sheaves,
cf.~\cite[Corollary A.6]{fuerdoes2025}.

\begin{Cor}
    Let $\sheaf{}$ be a semiregular quasianalytic sheaf satisfying \emph{\ref{Convolution}} and $P$
    be a differential operator with constant coefficients.
    Then $P$ is $\sheaf{}$-hypoelliptic if and only if $P$ is elliptic.
\end{Cor}
\begin{proof}
If $P$ is elliptic then $P$ is analytic hypoelliptic, see
e.g.~\cite{RodinoBook},
and therefore $\Set{u\in\Distr[\Omega]{n}\given Pu=0}\subseteq\analytic(\Omega)\subseteq\sheaf[\Omega]{n}$ for any open set $\Omega\subseteq\R[n]$, i.e.~$P$ is $\sheaf{}$-hypoelliptic according to
Theorem \ref{ThmCPDO1}.

On the other hand if $P$ is not elliptic then there has to exist some 
$\xi\in\R[n]\setminus\{0\}$ such that $p_d(\xi)=0$, where $p_d(\xi)$ is the principal symbol of $P=P(D)$. Let $N_\xi=\Set{y\in\R[n]\given y\xi=0}$ be the hyperplane normal to $\xi$ and let
\begin{equation*}
    H_\xi^+=\Set*{y\in\R[n]\given y\xi>0}\qquad\&\qquad
    H_\xi^-=\Set*{y\in\R[n]\given y\xi<0}
\end{equation*}
be the open half-spaces separated by $N_\xi$. Clearly $\overline{H_\xi^-}=H_\xi^-\cup N_\xi$.
According to \cite[Theorem 8.6.7]{HoeBook1} there exists some function $u\in\E_n(\R[n])$ such
that $Pu=0$ and $\supp u\subseteq \overline{H_\xi^-}$. In fact,  the proof of
\cite[Theorem 8.6.7]{HoeBook1} implies that $u$ does not vanish identically on $H_\xi^-$ but
of course $u\equiv 0$ on $H_\xi^+$. 
Therefore $u\notin\sheaf[{\R[n]}]{n}$ since the sheaf $\sheaf{}$ is quasianalytic,
cf.~subsection 2.1. 
Hence $P$ is not $\sheaf{}$-hypoelliptic by Theorem \ref{ThmCPDO1}.
\end{proof}

%One famous result on the regularity of differential operators with constant coefficients is that such operators are analytic hypoelliptic if and only if 
%they are elliptic. The theory developed in this section allows us to generalize this result from the analytic category to quasianalytic structures.
%\begin{Thm}
 %   Let $\sheaf{}$ be a quasianalytic semiregular sheaf satisfying \ref{Convolution} and $P$ be a differential operator with constant coefficients. 
  %  Then $P$ is $\sheaf{}$-hypoelliptic if and only if $P$ is elliptic.
%\end{Thm}
%\begin{proof}
%    From Theorem \ref{ThmCPDO1} we can infer that any $\sheaf{}$-hypoelliptic operator $P$ with constant coefficients is smooth hypoelliptic,
 %   since any fundamental solution of $P$ is of class $\sheaf{}$, and therefore smooth, outside the origin.
 %   By \cite{RodinoBook} there is some $s\geq 1$ such that $P$ is $s$-hypoelliptic, i.e.~$\G^s$-hypoelliptic where $\G^s$ is the Gevrey sheaf of degree $s$. Of course $\G^1=\analytic$ so we are done with $s=1$.
 %   According to \cite{RodinoBook} the index $s\geq 1$ is minimal, i.e.~the operator $P$ is not $t$-hypoelliptic for any $1\leq t<s$. In fact,
 %   there has to 
 %   be some $u_0\in\Distr[U]{n}$ such that $Pu_0=0$ but $u_0\notin\G^t(U)$
  %  for any $1\leq t< s$.
%\end{proof}
\begin{Rem}\leavevmode
\begin{enumerate}[label=(\arabic*),itemsep=0.2em]
        \item If $\sheaf{}$ is a microlocal sheaf then condition \ref{WFanalMap}
        implies that the following holds:
        \begin{enumerate}[label=(M\arabic*$^\prime$),start=4]
            \item\label{WFAnaDiff} If $\Phi:U\rightarrow V$ is a real-analytic diffeomorphism 
        then $\WFR \Phi^\ast(u)=\Phi^\ast(\WFR u)$ for any $u\in\Distr[V]{n}$.
        \end{enumerate}
        Using standard arguments, cf.~\cite[Chapter 8]{HoeBook1}, we can associate
        to each $u\in\Distr[M]{}$ defined on real-analytic manifolds $M$
        the ultradifferentiable wavefront set $\WFR u\subseteq S^\ast M$ defined
        as a closed subset of the cospherical bundle of $M$, which is a fiber bundle
        with typical fiber $S^{n-1}$.
        \item For a basic microlocal theory with respect to analytic differential operators it would be enough that $\WFR$ is invariant under analytic differential operators and diffeomorphisms. 
        Hence we may say that an isotropic ultradifferentiable sheaf $\sheaf{}$ is 
        \emph{weakly microlocal} if $\sheaf{}$ is semiregular and there exists a sheaf morphism
        $\WFR: \Distr{}\rightarrow \mathfrak{p}^{sc}$ satisfying \ref{BasicWF0}, \ref{WFComparision}, 
        \ref{WFSing}, \ref{WFInvolution}, \ref{WFAnaDiff}, \emph{\ref{WFsmoothing}} and
        \emph{\ref{AnalEllWF}}.
    \end{enumerate}
\end{Rem}

\begin{Rem}
 Suppose that $\sheaf{}$ is a non-quasianalytic semiregular sheaf satisfying \emph{\ref{Convolution}}.
 Then we observe that the space $\D_{\sheaf{}}(\R[n])=\sheaf[{\R[n]}]{n}\cap\D_n(\R[n])$ fulfills the hypothesis assembled in \cite[Definition 1.1]{Schapira68},
 cf.~Corollary \ref{WFCor3}.
 Then \cite[Proposition 1.1]{Schapira68} shows that there exists a locally convex topology on
 $\D_{\sheaf{}}(\R[n])$ such that
 \begin{itemize}[itemsep=0.3ex]
     \item The mappings $\E^\prime(\R[n])\times\D_{\sheaf{}}\rightarrow \D_{\sheaf{}}$, defined by convolution and multiplication
     $\D_{\sheaf{}}\times\D_{\sheaf{}}\rightarrow \D_{\sheaf{}}$ are separately continuous and the topology is the finest  locally convex topology.
     \item The topology is ultrabornological.
     \item If $\D_{\sheaf{}}(K)=\Set{\varphi\in\D_{\sheaf{}}\given \supp \varphi\subseteq K}$ for $K\subseteq \R[n]$ compact then
     \begin{equation*}
         \D_{\sheaf{}}=\varinjlim_{K\Subset \R[n]} \D_{\sheaf{}}(K)
     \end{equation*}
     as locally convex spaces.
 \end{itemize}
 For any open set $\Omega\subseteq\R[n]$ we define $\D_{\sheaf{}}(\Omega)=\varinjlim_{K\Subset\Omega}\D_{\sheaf{}}(K)$ as locally convex spaces and
 equip $\sheaf[\Omega]{n}$ with the
 coarest topology such that all the maps $f\rightarrow \varphi f$, 
 $\varphi\in\D_{\sheaf{}}$, are continuous. The dual spaces 
 $\D^\prime_{\sheaf{}}(\Omega)=(\D_{\sheaf{}}(\Omega))^\prime$
 and $\sheaf{n}^\prime(\Omega)=(\sheaf[\Omega]{n})^\prime$ are conferred with the
 Mackey topology. Then the presheaf $\D_{\sheaf{}}$ defined by 
 $\R[n]\supseteq \Omega\rightarrow\D^\prime_{\sheaf{}}(\Omega)$ is a
 soft sheaf of vector spaces \cite[Proposition 1.2]{Schapira68} which is actually
 a subsheaf of the sheaf of hyperfunctions \cite[Proposition 6.1]{Schapira68}.
 Moreover, the set of compact sections of $\D^\prime_{\sheaf{}}(\Omega)$
 can be identified with $\sheaf{n}^\prime(\Omega)$ \cite[Proposition 1.2(iii)]{Schapira68}.
\end{Rem}

\section{Geometric sheaves of smooth functions}\label{Section3}
\subsection{Regular sheaves}
\begin{Def}
An semiregular ultradifferentiable sheaf $\sheaf{}$ is regular if the following conditions are satisfied:
\begin{enumerate}[label=(R\arabic*),itemsep=0.2em]
    \item\label{ClosedComp} $\sheaf{}$ is closed under composition: Let $U\subseteq\R[n]$, $V\subseteq\R[m]$ be open sets
    and $\Phi: U\rightarrow V$ be a mapping of class $\sheaf{}$, i.e.~ $\Phi=(\Phi_1,\dotsc,\Phi_m)$
    and $\Phi_j\in\sheaf[U]{n}$.
    Then
    \begin{equation*}
        g\circ\Phi\in\sheaf[U]{n}
    \end{equation*}
    for all $g\in\sheaf[V]{m}$.
    \item\label{InverseThm} The inverse function theorem holds in $\sheaf{}$: Let 
    $U,V\subseteq\R[n]$ be open sets and $F: U\rightarrow V$
    be a mapping of class $\sheaf{}$.
    If the Jacobian 
    \begin{equation*}
        \frac{\partial F}{\partial x}(p)=
        \begin{pmatrix}
\frac{\partial F_1}{\partial x_1}(p) & \cdots &\frac{\partial F_1}{\partial x_n}(p)\\
\vdots & &\vdots\\
\frac{\partial F_n}{\partial x_1}(p) &\cdots &\frac{\partial F_n}{x_n}(p)
        \end{pmatrix}
    \end{equation*}
    at some point $p\in U$ is invertible then $F$ is a diffeomorphism of class $\sheaf{}$ near $p$. 
    %there are neighborhoods
    %$U^\prime$ of $p$ and $V^\prime$ of $q=F(p)$ and
    %a mapping $G: V^\prime\rightarrow U^\prime$ of class $\sheaf{}$
    %such that $G(q)=p$ and $F\circ G=\Id_{V^\prime}$.
    \item\label{ODEclosed} $\sheaf{}$ is closed under solving ODEs:
    Suppose that $f\in\sheaf[U\times I]{n+1}$ with $U\subseteq\R[n]$, $I\subseteq\R[]$ being open sets.
    If we consider the initial value problem
    \begin{equation}\label{IVP}
        x^\prime(t)=f(x(t),t), \quad x(t_0)=y_0, \qquad t_0\in I, \; y_0\in U
    \end{equation}
    then for each pair $(y_0,t_0)\in U\times I$ there are neighborhoods $V\subseteq\R[n]$ and $J\subseteq \R[]$ of $y_0$ and $t_0$, respectively, and  a mapping $x: J\times V\rightarrow V$
    of class $\sheaf{}$ such that $x_y=x(\,.\,,y)$ satisfying \eqref{IVP}.
\end{enumerate}
\end{Def}

\begin{Rem}
    Due to the smooth inverse function theorem it is enough to show that the inverse function $F^{-1}$ is of class $\sheaf{}$ near $F(p)$ in order to establish \ref{InverseThm}.
\end{Rem}
We leave the following characterization of mappings of class $\sheaf{}$ to the reader.
\begin{Lem}\label{CharactMapping}
    Let $\sheaf{}$ be an isotropic sheaf which is closed under composition and $U\subseteq\R[n]$, $V\subseteq\R[m]$ two open sets.
    If $\Phi=(\varphi_1,\dotsc,\varphi_m):\, U\rightarrow V$ is a mapping then
    the following statements are equivalent:
    \begin{enumerate}
        \item $\varphi_j\in\sheaf[U]{n}$ for all $j=1,\dotsc,m$.
        \item $\chi\circ \Phi\in\sheaf[U]{n}$ for all $\chi\in\sheaf[V]{m}$.
    \end{enumerate}
    \begin{center}
        \begin{tikzcd}
            U\arrow[rr,"\Phi"]\arrow[dr,"\chi\circ\Phi"']& &V\arrow[dl,"\chi"]\\
            & \mathbb{C} &
        \end{tikzcd}
    \end{center}
\end{Lem}

\begin{Thm}
    Let $\sheaf{}$ be a semiregular sheaf. Then the inverse function theorem holds
    in $\sheaf{}$, i.e.~\emph{\ref{InverseThm}} holds in $\sheaf{}$, if and only if the implicit function theorem holds in $\sheaf{}$:
    \begin{enumerate}[label=\emph{(R{\arabic*}')},start=2]
        \item\label{ImplicitFunctThm}
        Let $U\subseteq\R[n+m]$ (with coordinates $(x,y)=(x_1\dotsc,x_n,y_1,\dotsc,y_m)$)
        and $\Psi: U\rightarrow \R[m]$ a mapping of class $\sheaf{}$.
        If $\Psi(x_0,y_0)=0$ and $\partial \Psi(x_0,y_0)/\partial y$ is invertible at some point $(x_0,y_0)\in U$ then there are open neighborhoods $V\subseteq\R[n]$ of $x_0$ and
        $W\subseteq\R[m]$ of $y_0$ and a (unique) mapping $\varphi: V\rightarrow W$ of class $\sheaf{}$ such
        that $V\times W\subseteq U$, $\varphi(x_0)=y_0$ and
        \begin{equation*}
            \Psi(x,\varphi(x))=0,\qquad \fa x\in V.
        \end{equation*}
        %\item The sheaf $\sheaf{}$ is closed under inversion: If $f\in\sheaf[n]{U}$ is such
        %that $f(x)\neq 0$ for all $x\in U$ then $1/f\in\sheaf[n]{U}$.
    \end{enumerate}
    Moreover, if \ref{ImplicitFunctThm} holds then the following holds:
    \begin{enumerate}[label=\emph{(R\arabic*)},start=4]
        \item\label{Inversion} 
        The sheaf $\sheaf{}$ is closed under inversion: If $f\in\sheaf[U]{n}$ is such
        that $f(x)\neq 0$ for all $x\in U$ then $1/f\in\sheaf[U]{n}$.
    \end{enumerate}
\end{Thm}
\begin{proof}
Assuming that \ref{InverseThm} holds and that $\Psi$ is a function satisfying the assumptions of
\emph{\ref{ImplicitFunctThm}}.
We may note that we can follow directly the classical proof in the smooth setting,
see e.g.~\cite{RudinBook}:
Setting as usual $F(x,y)=(x,\Psi(x,y))$ it is easy to see that
the Jacobi matrix of $F$ is invertible at $p_0=(x_0,y_0)$. 
Moreover, $F$ is clearly of class
$\sheaf{}$ on $U$, cf.~Lemma \ref{CharactMapping}.
Thus \ref{InverseThm} implies that there are open sets $U_1\subseteq\R[n]$,
$U_2\subseteq\R[m]$ and $V\subseteq\R[n+m]$ with
$x_0\in U_1$, $y_0\in U_2$ and $(x_0,0)\in V$ 
such that $F: U_1\times U_2\rightarrow V$ is a diffeomorphism
of class $\sheaf{}$. 
Hence there is a (small) neighborhood $W$ of $x_0$ such that
for each $x\in W$ there exists a (unique) point $y\in U_2$ such that $(x,0)=F(x,y)$. If $\Psi^{-1}=(\psi_1,\psi_2)$, where $\psi_j: \, V\rightarrow U_j$, $j=1,2$, are the components of $\Psi^{-1}$, then $\varphi: W\rightarrow U_2$
given by $\varphi(x)=\psi_2(x,0)$ is the desired function, which is obviously of 
class $\sheaf{}$.

On the other hand, suppose that
\emph{\ref{ImplicitFunctThm}} holds and $F$ satisfies the hypothesis of \ref{InverseThm}.
Then $\Psi(x,y)=x-F(y)$ fulfills the assumptions of \emph{\ref{ImplicitFunctThm}}
and therefore there is a mapping $G: V\rightarrow U$ of class $\sheaf{}$
such that $\Psi(x,y)=0$ for $(x,y)\in V\times U$ if and only if $y=G(x)$. Hence $F(G(x))=x$ and since obvious $\Psi(F(y),y)=0$ we have also $y=G(F(x))$.
Therefore $F$ is locally a diffeomorphism of class $\sheaf{}$ near $p$.

    In order to prove \emph{\ref{Inversion}}
    note that it is enough to show the assertion for real-valued functions since if $g\in\sheaf[U]{n}$ is non-zero at each point then 
    \begin{equation*}
        \frac{1}{g(x)}=\frac{\overline{g(x)}}{\abs{g(x)}^2}, \qquad x\in V
    \end{equation*}
    and $\abs{g}^2=g\overline{g}\in\sheaf[U]{n}$.
    
    So let $f\in\sheaf[U]{n}$ be real-valued such that $f(x)\neq 0$ for all $x\in U$.
    Then the mapping $F(x,y)=f(x)y-1$ is of class $\sheaf{}$ in $U\times\R[]$ by \ref{Tensor}.
    Since
    \begin{equation*}
        \frac{\partial F}{\partial y}(x,y)=f(x),\qquad (x,y)\in U\times\R[]
    \end{equation*}
    and $F(x_0,y_0)=0$ where $y_0=1/f(x_0)$, we see that the assumptions of \emph{\ref{ImplicitFunctThm}} are
    satisfied and therefore for each $x_0\in U$ there is a neighborhood $V$ such that
    $1/f\in\sheaf[V]{n}$. It follows that $1/f\in\sheaf[U]{n}$
\end{proof}
\subsection{Differential geometry in the ultradifferentiable category}
If we consider a regular sheaf $\sheaf{}$, then we are able to define a theory of manifolds of class $\sheaf{}$ mirroring the basic theory in the smooth and analytic case:

\begin{Def}[Manifolds of class $\sheaf{}$]
W say that a smooth manifold $M$ is of class $\sheaf{}$
if there is an atlas $\atlas_M$ of $M$ such that for any two charts $\varphi_1$, $\varphi_2\in\mathcal{A}$ we have that
$\varphi_1\circ\varphi_2^{-1}$ is a map of class $\sheaf{}$ whenever defined.

If $M,N$ are two manifolds of class $\sheaf{}$ with atlases $\atlas_M$, $\atlas_N$ and $\Phi: M\rightarrow N$ a mapping then
we say that $\Phi$ is a map of class $\sheaf{}$ if for all $\varphi\in\atlas_M$, $\psi\in\atlas_N$
the map
\begin{equation*}
    \psi\circ\Phi\circ\varphi^{-1}
\end{equation*}
is of class $\sheaf{}$ wherever defined.
\end{Def}

\begin{Rem}
    Regarding smooth manifolds we are going to work in the setting of \cite{BCHbook}; in particular we assume that
    a smooth manifold is paracompact.
\end{Rem}
We may consider manifolds $M$ of class $\sheaf{}$ as ringed spaces $(M,\sheaf{M})$
where $\sheaf{M}$ is the sheaf of functions of class $\sheaf{}$ over $M$, cf.~\cite{BierstoneMilman}.

The proof of the following theorem is just a verbatim repetition of the proof of the well-known statement in the smooth category applying \emph{\ref{ImplicitFunctThm}}.
\begin{Thm}\label{Submanifold}
Let $M\subseteq\R[n]$ be a subset and $\sheaf{}$ a regular sheaf. Then the following statements are equivalent:
\begin{enumerate}[label=(\roman*),itemsep=0.4ex]
    \item $M$ is a submanifold of $\R[n]$ of class $\sheaf{}$ of dimension $m$.
\item For each point $p\in M$ there is a neighborhood $\Omega\subseteq\R[n]$ of $p$ and a mapping
    $\Phi: \Omega\rightarrow \R[n-m]$ of class $\sheaf{}$ such that
    $d\Phi(p)\neq 0$ for all $p\in\Omega$ and 
    \begin{equation*}
        M\cap\Omega=\Phi^{-1}(0)=\Set*{q\in\Omega\given \Phi(q)=0}.
    \end{equation*}
    \end{enumerate}
    The number $d=n-m$ is then called the codimension of $M$.
\end{Thm}

\begin{Def}[Bundles of class $\sheaf{}$]
Suppose that $\sheaf{}$ is a regular sheaf and $M$ is a manifold of class $\sheaf{}$.
\begin{enumerate}[label=(\arabic*),itemsep=0.4em]
    \item A fiber bundle of class $\sheaf{}$ over $M$ with fiber $F$, $F$ being itself a manifold of class $\sheaf{}$, is a manifold of class $\sheaf{}$ such that the following holds
    \begin{enumerate}[label=(\roman*),itemsep=0.1em]
        \item There is a map $\pi: V\rightarrow M$ of class $\sheaf{}$, called the projection.
        \item For each point $p\in M$ there are a neighborhood $U\subseteq M$ of $p$ and a diffeomorphism $\varphi: \pi^{-1}(U)\rightarrow U\times F$ such that the following diagram commutes
        \begin{center}
            \begin{tikzcd}
                \pi^{-1}(U) \arrow[r,"\varphi"]\arrow[d, "\pi"'] & U\times F\arrow[dl]\\
                U &
            \end{tikzcd}
        \end{center}
    \end{enumerate}
    \item A real (resp.~complex) vector bundle of class $\sheaf{}$ and fiber dimension $N$ over $M$ is a fiber bundle
    $E$ over $M$ with fiber $\R[N]$ (resp.~$\C[N]$) with the following properties
    \begin{enumerate}[label=(\roman*),itemsep=0.1ex]
\item On each fiber $V_p=\pi^{-1}(p)$, $p\in M$, there is a vector space structure (over $\mathbb{K}=\R[],\C[]$).
 \item If $p,q\in M$ and if $U_p$ and $U_q$ are neighborhoods of $p$ and $q$, respectively, such that $U_p\cap U_q\neq\emptyset$ then there is a mapping
    $g_{pq}: U_p\cap V_q\rightarrow\GL(N,\mathbb{K})$ of class $\sheaf{}$ such that
    \begin{enumerate}[label=\greek*)]
\item $g_{pp}=\Id$
\item $g_{xy}g_{yz}=g_{xz}$ wherever defined.
\item If $X\in \pi^{-1}(U_p)\cap \pi^{-1}(U_q)$ and $(x,e)=\psi_p(X)$ then
$\psi_q\circ\psi_p^{-1}((x,e))=(x, g_{pq}(x)e)$.
    \end{enumerate}
    \end{enumerate}
\end{enumerate}
The following Proposition is proved in the same manner as in the smooth category.
The details are left for the reader.
\begin{Prop}
    If $M$ is a manifold of class $\sheaf{}$ of dimension $n$ then the tangential bundle $TM$ is a $\R$-vector bundle of class $\sheaf{}$ with fiber dimension $n$ whereas the complexified tangent bundle $\C TM$, with fibers $\C T_pM=\C\otimes T_pM$, is a $\C$-vector bundle
    of class $\sheaf{}$.
\end{Prop}
%If $M$ is a manifold of class $\sheaf{}$ then a $\mathbb{K}$-vector bundle of fiber dimension $N$ is
%a manifold $V$ of class $\sheaf{}$ with the following properties:
%\begin{enumerate}[label=(\roman*),itemsep=0.1em]
  %  \item There is a map $\pi: V\rightarrow M$ of class $\sheaf{}$, called the projection.
 %   \item On each fiber $V_p=\pi^{-1}(p)$, $p\in M$, there is a vector space structure (over $\mathbb{K}$).
  %  \item For each $p\in M$ there are an open set $U_p$ in $M$,  $p\in M$,  and a diffeomorphism
 %   $\psi_p: p^{-1}(U_p)\rightarrow U_p\times\mathbb{K}^N$ of class
%    $\sheaf{}$ such that the following diagram commutes:
 %   \begin{center}
%    \begin{tikzcd}
 %p^{-1}(U_p)\arrow[r,"\psi_p"]\arrow[d,"p"] & U_p\times\mathbb{K}^N\arrow[d,"\pr_1"]\\
%U_p\arrow[r,"\Id"] &U_p
 %   \end{tikzcd}
 %   \end{center}
   % \item If $p,q\in M$ and $U_p\cap U_q=\emptyset$ then there is a mapping
  %  $g_{pq}: U_p\cap V_q\rightarrow\GL(N,\mathbb{K})$ of class $\sheaf{}$ such that
   % \begin{enumerate}[label=\greek*)]
%\item $g_pp=\Id$
%\item $g_{xy}g_{yz}=g_{xz}$ wherever defined.
%\item If $X\in p^{-1}(U_p)\cap p^{-1}(U_q)$ and $(x,e)=\psi_p(X)$ then
%$\psi_q\circ\psi_p^{-1}((x,e))=(x, g_{pq}e)$.
  %  \end{enumerate}
%\end{enumerate}
\end{Def}
A section of class $\sheaf{}$ of the vector bundle $E$ is a mapping $X: M\rightarrow E$ such that
$p\circ M=\Id_M$. We denote the sheaf of sections of class $\sheaf{}$ by $\sheaf{E}$ and
by $\Distr{E}$ the sheaf of distributional sections of $E$.

In particular, a section $X\in\VFS[U]=\mathfrak{X}_{\sheaf{}}(U)=\sheaf[U]{TM}$ is a vector field of class $\sheaf{}$ defined on the open set $U\subseteq M$.
Note that $\VFS[U]$ is a (real) Lie algebra with the Lie bracket
\begin{equation*}
    [X,Y]=XY-YX \qquad X,Y\in\VFS[U]
\end{equation*}
If $X\in\VFS[U]$ and $p\in U$ then the integral curve $c_p$ through $p$ is defined as the solution
of the initial value problem
\begin{equation*}
    c_p^\prime(t)=X\bigl(c_p(t)\bigr)\qquad x_p(0)=p. 
\end{equation*}
It follows from \ref{ODEclosed} (and \ref{ClosedComp}) that $c_p$ is a curve of class
$\sheaf{}$ on its domain of definition.
Adapting the arguments in the smooth case in the obvious way we see that the local flow
$\Fl_X$ of $X$ is of class $\sheaf{}$ on its domain of definition when $X$ is of class $\sheaf{}$.
\begin{Def}
Following \cite{TrevesBook25} we say that, given some subset $\mathcal{A}\subseteq\VFS[M]$, that a subset $N$ of $M$ is
an integral manifold of class $\sheaf{}$ of $\mathcal{A}$ if
$N$ is an immersed submanifold of class $\sheaf{}$\footnote{Which is defined in the obvious way} of $M$ and satisfies the following properties:
\begin{enumerate}[label=(\roman*)]
    \item For each $p\in N$ we have that $\mathcal{A}_p=\Set{X(p)\in T_pM \given X\in\mathcal{A}}=T_p N$
    \item If $N^\prime$ is another immersed submanifold of $M$ satisfying (i) then
    either $N^\prime\subseteq N$ or $N\cap N^\prime=\emptyset$.
\end{enumerate}
\end{Def}

We are now at a point where we can formulate a quasianalytic version of the Nagano theorem \cite{NaganoPaper}.

\begin{Thm}\label{NaganoThm}
    Let $\sheaf{}$ be a quasianalytic regular sheaf and $M$ be a manifold of class $\sheaf{}$.
    If $\mathfrak{g}\subseteq\mathfrak{X}_{\sheaf{}}$ is a Lie subalgebra
    then the integral manifolds of $M$ form a foliation of class $\sheaf{}$ of $M$, i.e.~for each point $p\in M$ there is an (unique) integral manifold
    $N$ of class $\sheaf{}$ of $\mathfrak{g}$ such that $p\in N$.
\end{Thm}
The proof is principally a verbatim repetition of the proof given in e.g.~\cite{BERbook}, cf.~\cite{FuerdoesCR}. For a detailed proof in the case
of quasianalytic Denjoy-Carleman classes, see \cite{FuerdoesThesis}.
\begin{Rem}
 In the smooth category the Theorem of Sussmann \cite{Sussmann} is a weak 
 analogue of the Nagano theorem. It is possible to prove
 an ultradifferentiable version of the Sussmann theorem following e.g.~the proof
 in \cite{BERbook}. For a proof in the Denjoy-Carleman setting we refer again
 to \cite{FuerdoesThesis}.
\end{Rem}

\begin{Rem}
Bierstone and Milman \cite{BierstoneMilman} showed that 
a quasianalytic regular sheaf $\sheaf{}$ admits resolution of singularities,
which they used, for example, to prove {\L}ojasiewicz inequalities for the elements
of $\sheaf[U]{n}$, for this and other applications see \cite[section 6]{BierstoneMilman}.
\end{Rem}

%If $E$ is a $\mathbb{K}$-vector bundle of class $\sheaf{}$ over an 

We are in particular interested on differential operators acting on
vector bundles. We briefly summarize the most basic facts, cf.~\cite{HoeBook1}.
Let $E$ and $F$ be two complex vector bundles over a manifold $M$, all three of class 
$\sheaf{}$. A linear operator $P:\sheaf[M]{E}\rightarrow\sheaf[M]{E}$
is a differential operator of class $\sheaf{}$ if locally $P$ is a differential operator, i.e.~let $U$ be a coordinate neighborhood in $M$ with local coordinates
$x=(x_1,\dotsc,x_n)$ and $(e_1,\dotsc,e_N)$
and $(f_1,\dotsc,f_M)$ be local frames of $E$ and $F$, respectively, over $U$ then
\begin{equation*}
    (Pu)_j=\sum_{k=1}^N P_{jk}u_k,\qquad j=1,\dotsc,M, 
\end{equation*}
where $u\vert_U=\sum_{k=1}^N u_ke_k$, $(Pu)\vert_{U}=\sum_{j=1}^M (Pu)_j f_j$
and $P_{jk}$ is a differential operator in $U$.
This construction is independent of the choices of local coordinates and local frames. We say that $P$ is of order $\leq d$ if we can write
\begin{equation*}
    P_{j,k}=\sum_{\abs{\alpha}\leq d}p_{jk,\alpha}(x)D^\alpha, \qquad 
    p_{jk,\alpha}\in\sheaf[U]{n},\qquad\fa j,k,
\end{equation*}
and $P$ is of order $d$ if $P$ is not of order $\leq d-1$.
We denote the homogeneous part of degree $d$ of $P_{jk}$ by 
\begin{equation*}
    p_{jk}^{d}(x,\xi)=\sum_{\abs{\alpha}=d} p_{jk,\alpha}(x)\xi^\alpha 
\end{equation*}
which may vanish even if $P$ is of order $d$. 
The principal symbol $p$ of a differential operator $P$ of order $d$ is locally
of the form
\begin{equation*}
    p^d(x,\xi)=\begin{pmatrix}
        p^d_{11}(x,\xi)&\dots & p^d_{1N}(x,\xi)\\
        \vdots & \ddots & \vdots\\
        p^d_{M1}(x,\xi)&\dots & p^d_{MN}(x,\xi).
    \end{pmatrix}
\end{equation*}
This definition is again independent of the choices of local coordinates and frame,
see \cite{HoeBook1} for a coordinate-free definition.

From now on we assume that $N=M$, i.e.~the bundles $E,F$ are of the same rank.
The characteristic set of a differential operator $P:\sheaf[M]{E}\rightarrow\sheaf[M]{F}$ is defined by
\begin{equation*}
    \Char P=\Set*{\gamma\in S^\ast M\given p(\gamma)\text{ is not invertible}}\subseteq S^\ast M.
\end{equation*}

If we have an ultradifferentiable class $\sheaf{}$ which is both microlocal and regular, then we would like to define $\WFR u$
for distributions $u$ on manifolds of class $\sheaf{}$.
For this we need that $\WFR$ is invariant under diffeomorphisms of class $\sheaf{}$.
We would, of course, also like to define $\WFR u$ for distributions with values in some vector bundle.
As a first step we define the wavefront set of a vector-valued distribution  $u=(u_1,\dotsc,u_m)\in\Distr[U;{\C[m]}]{n}\cong(\Distr[U]{n})^m$ by setting
\begin{equation*}
    \WFR u=\bigcup_{j=1}^m\WFR u_j.
\end{equation*}

\begin{Def}
    An ultradifferentiable sheaf $\sheaf{}$ is normal if it is microlocal and regular and
    the following invariance conditions are satisfied:
    \begin{enumerate}[label=(N\arabic*)]
\item\label{UltraWFInv0} If $\Phi: U\rightarrow V$ is a mapping of class $\sheaf{}$  
Then $\Phi^\ast u\in \Distr[U]{n}$ is well-defined and 
\begin{equation*}
            \WFR\left(\Phi^\ast u\right)\subseteq\Phi^\ast\left(\WFR u\right)=
            \Set*{\left(x,\frac{\prescript{t}{}{\Phi}^\prime(x)\eta}{\abs*{\prescript{t}{}{\Phi}^\prime(x)\eta}}\right)
            \in S^\ast U\given (\Phi(x),\eta)\in S^\ast V}
        \end{equation*}
        for $u\in\Distr[V]{m}$ with $\WFR u\cap N_\Phi=\emptyset$.
\item\label{VectorMicroEll} Let $P=(P_{jk})_{jk}$, $1\leq j,k\leq n$, be a differential operator acting on
vector-valued distributions $u\in{\Distr[U;{\C[m]}]{n}}$ and $\sheaf{}$ a normal ultradifferentiable structure. Then
\begin{equation*}
    \WFR Pu\subseteq\WFR u\subseteq\WFR Pu\cup \Char P.
\end{equation*}
    \end{enumerate}
\end{Def}
\begin{Thm}\label{UltraDiffeoWF}
If $\sheaf{}$ is a normal sheaf then the following statement holds:
\begin{enumerate}[label=\emph{(N\arabic*$^{\prime}$)}]
    \item\label{UltraWFinvar} If $\Phi: U\rightarrow V$ is a diffeomorphism of class $\sheaf{}$
then 
\begin{equation*}
    \Phi^\ast\left(\WFR u\right)=\WFR \Phi^\ast u,\qquad \fa u\in{\Distr[U]{n}}.
\end{equation*}
\end{enumerate}
\end{Thm}
\begin{Rem}
    Using Theorem \ref{UltraDiffeoWF} we are able to define $\WFR u$ for
    distributions $u\in\Distr[M]{n}$ and distributional sections $u\in\Distr[M]{E}$.
    Therefore, given a differential operator $\Distr[M]{E}\rightarrow\Distr[M]{F}$
    of class $\sheaf{}$, where $E,F$ are vector bundles of the same rank, then
    \ref{VectorMicroEll} gives
    that
    \begin{equation*}
        \WFR u\subseteq\WFR Pu\cup\Char P\qquad \fa u\in\Distr[M]{E}.
    \end{equation*}
\end{Rem}

\begin{Rem}
    If $\sheaf{}$ is a microlocal sheaf then following the lines of Theorem \ref{MicroEllThm1} we could prove \ref{VectorMicroEll} (for at least scalar operators) if we knew that an elliptic differential operator of class
    $\sheaf{}$ has an ultradifferentiable parametrix. However, this would apparently require
    to use the defining data of the sheaf. In the Denjoy-Carleman category a comparison of the conditions on the weight sequence needed in
    the construction of the parametrix in \cite{Matsumoto} with the conditions
    used in \cite{HoeBook1} to prove properties \ref{WFComparision}, \ref{WFproj}
    and \ref{WFTensor} indicates that there is a fundamental difference between
    \ref{VectorMicroEll} and \emph{\ref{AnalEllWF}}.
    In fact, the conditions required in \cite{Matsumoto}  mean   that the class is non-quasianalytic and that it
    is contained in some Gevrey class, cf.~\cite[Proposition 2.6]{Matsumoto},  whereas 
    the conditions used in \cite{HoeBook1} amount to the semiregularity of the class under consideration.
    
    In the notions of Section \ref{Section5}: The weight sequence in \cite{Matsumoto}
    is non-quasianalytic and essentially normal\footnote{It does not need to be log-convex if we do not consider operators on manifolds.}, cf.~Definition \ref{NormalMatrix} and the remarks afterwards, whereas
    the sequence in \cite{HoeBook1} only needs to be semiregular, cf.~Definition \ref{SemiMatrix}).
\end{Rem}
\begin{Def}
    A regular sheaf $\sheaf{}$ is weakly normal if it is also 
    weakly microlocal and
    satisfies \emph{\ref{UltraWFinvar}} and
    \begin{enumerate}[label=(N\arabic*$^\prime$),start=2]
        \item\label{ScalarUltraEll} If $P$ is a differential operator with coefficients in $\sheaf[U]{n}$ then
        \begin{equation*}
            \WFR u\subseteq\WFR Pu\cup \Char P\qquad \fa u\in\Distr[U]{n}.
        \end{equation*}
    \end{enumerate}
\end{Def}
\subsection{Uniqueness theorems for quasianalytic operators} 
In this section we assume that $\sheaf{}$ is a normal sheaf which satisfies 
the following microlocal uniqueness property
\begin{enumerate}[label=(N\arabic*),start=3]
    \item\label{QuasiWF} Let $I\subseteq\R$ be an open interval and $u\in\Distr[I]{1}$.
    If $x_0\in\partial\supp u$ then $(x_0,\pm 1)\in\WFR u$.
\end{enumerate}
Observe that a microlocal sheaf $\sheaf{}$ which satisfies \ref{QuasiWF} is automatically quasianalytic, since if $\sheaf{}$ is non-quasianalytic then
any non-zero $u\in\sheaf[I]{1}\cap\D_1(I)$ violates \ref{QuasiWF}.
Adapting the proof of H\"ormander in the analytic setting \cite[Theorem 8.5.6]{HoeBook1}
cf.~\cite[Theorem 7.2]{Fuerdoes20} for the Denjoy-Carleman version, we can extend
\ref{QuasiWF} to the higher dimensional case.
\begin{Thm}\label{QuasiWFThm1}
Let $\Omega\subseteq\R[n]$, $u\in\Distr[\Omega]{n}$, $x_0\in\supp u$ and $f:\Omega\rightarrow \R$ be a function of class $\sheaf{}$ such that
$df(x_0)\neq 0$ and $f(x)\leq f(x_0)$ for $x_0\neq x\in\supp u$.
Then
\begin{equation*}
   \left(x_0,\pm \frac{df(x_0)}{\abs{df(x_0)}}\right)\in \WFR u.
\end{equation*}
\end{Thm}
\begin{proof}
Replacing $f$ by the function $f(x)-\abs{x-x_0}^2$ of class $\sheaf{}$ we can assume
that $f(x)<f(x_0)$ for $x_0\neq x\in\supp u$.
Moreover, since $\sheaf{}$ is a normal sheaf, i.e.~is in particular regular
and satisfies \ref{UltraDiffeoWF}, and $df(x_0)$ we can change coordinates
such that $x_0=0$ and $f(x)=x_n$, possibly shrinking $\Omega$.

We continue by choosing a neighborhood $V$ of $0\in\R[n-1]$ such that
$V\times\{0\}\Subset \Omega$. 
The hypothesis implies that $\supp u\cap \overline{V}\times \{0\}=\{0\}$.
Hence we can find an open interval $I$ containing $0\in\R$ such that
\begin{equation}\label{Uniq1}
    V\times I\Subset \Omega\qquad \&\qquad \supp u\cap \bigl(\partial V\times I\bigr)=\emptyset.
\end{equation}
If $A$ is an entire real-analytic function in the variables 
$x^\prime=(x_1,\dotsc,x_{n-1})\in\R[n-1]$ we define a distribution $u_A\in\Distr[I]{1}$ 
by setting $\langle u_A,\psi\rangle=\langle u,A\otimes\psi\rangle$ 
which is well-defined because of \eqref{Uniq1}. Proposition \ref{WFProp2} implies that
\begin{equation*}
\WFR u_A\subseteq\Set*{(x_n,\xi_n)\in S^\ast I\given \ex x^\prime\in V:\;(x^\prime,x_n,0,\xi_n)\in\WFR u }.    
\end{equation*}
We may note that on the right-hand side above that $(x^\prime,x_n)$ has to be close
to the origin for $x_n$ small.

If we assume that e.g. $(0,e_n)\notin\WFR u$, with $e_n=(0,\dotsc,0,1)\in\R[n]$,
then we can choose $I$ such that $(x,e_n)\notin\WFR u$ for $x\in V\times I$.
It follows that $(x_n,1)\notin\WFR u_A$ for $x_n\in I$ and \ref{QuasiWF} gives that
$U_A=0$ since  $u_A=0$ on the halfspace $\{x_n>0\}$. That means that
\begin{equation*}
    \Bigl\langle u\vert_{V\times I},A\otimes\psi\Bigr\rangle=0
\end{equation*}
for all $A$ entire real-analytic functions and every $\psi\in\D_1(I)$. Since the space of entire real-analytic functions is dense in $\E_{n-1}(\R[n-1])$ we obtain that $u=0$ in $V\times I$
\end{proof}
Following H\"ormander's arguments in \cite{HoeBook1}, see also \cite{Fuerdoes20},
we can prove a $\sheaf{}$-version of Holmgren's Theorem (note that $\sheaf{}$ is
assumed to be quasianalytic in this section).
\begin{Def}[cf.~{\cite[Definition 8.5.7]{HoeBook1}}]
Assume that $F$ is a closed subset of a $\mathcal{C}^2$-manifold $X$ then the 
exterior normal set $N_e(F)\subseteq S^\ast X$ is defined as the set of all
$(x_0,\xi_0)\in S^\ast X$ such that $x_0\in F$ and there is a real-valued function
$f\in\mathcal{C}^2(X)$ with $df(x_0)=\xi_0$ and $f(x)\leq f(x_0)$ for $x\in F$.
\end{Def}
Note that by the remarks in \cite[p.~300]{HoeBook1} $f$ has only to be defined
near $x_0$ and may be chosen to be real-analytic. Clearly if $(x_0,\xi_0)\in N_e(F)$ then $x_0\in\partial F$. In fact, if $\pi: S^\ast X\rightarrow X$ is the
canonical projection then according to \cite[Proposition 8.5.8]{HoeBook1} 
$\pi(N_e(F))$ is dense in $\partial F$ and if $g\in\mathcal{C}^1(X)$ with
$dg(x_0)=\xi_0\neq 0$ and $g(x)\leq g(x_0)$ for $x\in F$ then
\begin{equation*}
    \left(x_0,\frac{\xi_0}{\abs{\xi_0}}\right)\in \overline{N_e(F)},
\end{equation*}
since we can approximate $g$ by smooth functions in $\mathcal{C}^1$.
Moreover the interior normal set of $F$ is given by $N_i(F)=\Set{(x,\xi)\in S^\ast X\given (x,-\xi)\in N_e(F)}$ and thus we call $N(F)$ the normal set of $F$.
Then Theorem \ref{QuasiWFThm1} can be reformulated:
\begin{Thm}\label{QuasiWFThm2}
Let $\Omega\subseteq\R[n]$. Then
\begin{equation*}
    \overline{N(\supp u)}\subseteq\WFR u\qquad \fa u\in\Distr[\Omega]{n}.
\end{equation*}
\end{Thm}
Combining Theorem \ref{QuasiWFThm2} with \ref{ScalarUltraEll} gives
that
\begin{Cor}
    Let $P$ be a differential operator of class $\sheaf{}$ in $\Omega$ and
    $u\in\Distr[\Omega]{n}$ such that $Pu=0$ in $\Omega$. Then
    \begin{equation*}
        \overline{N(\supp u)}\subseteq\Char P.
    \end{equation*}
\end{Cor}
    In order to formulate the quasianalytic Holmgren Theorem we recall
    that a $\mathcal{C}^1$-hypersurface $X\subseteq\Omega$ is characteristic
    at some point $x\in X$ for a linear differential operator $P$ in $\Omega$
    if there is a defining function $\varphi$ of $X$ near $x$ such that
    $(x, d\varphi(x))\in\Char P$.
\begin{Thm}\label{QuasiHolmgren}
    Let $P$ be a differential operator of class $\sheaf{}$ in $\Omega$.
    If $X\subseteq\Omega$ is a non-characteristic hypersurface at $x_0$
    and $u\in\Distr[\Omega]{n}$ is a solution of $Pu=0$ which vanishes on one side
    of $X$ near $x_0$ then $u$ vanishes on a full neighborhood of $x_0$.
\end{Thm}
Obviously we can formulate Theorem \ref{QuasiHolmgren} for differential operators defined on manifolds of class $\sheaf{}$ and, using \ref{VectorMicroEll}, we
can also consider differential operators acting on sections of vector bundles,
cf.~\cite[Remark 7.8]{Fuerdoes20}.

Following \cite[Section 7]{Fuerdoes20} we could continue with proving
extensions of uniqueness statements for analytic differential operators from
e.g.~\cite{Bony69}, \cite{Bony76}, \cite{Zachm69} and \cite{Zachm72}.
However this would require a deep dive in the geometric theory of linear differential operators which is beyond the scope of this article.
Here we only formulate the following statement for sum of squares operators,
which proof can easily be adapted from the proof of \cite[Corollary 7.15]{Fuerdoes20}, using Theorem \ref{NaganoThm}.
Recall, cf.~\cite{Hoe67}, that a collection $Q_1,\dotsc,Q_N$ of real-valued vector 
fields defined on $\Omega\subseteq\R[n]$ is of finite type if the Lie algebra generated by $Q_1,\dotsc,Q_N$ span $\R^n$ at each point of $\Omega$.
\begin{Thm}
    Suppose that $Q_1,\dotsc,Q_N$ are a collection of real vector fields of class $\sheaf{}$, which are of finite type in an open connected set 
    $\Omega\subseteq\R[n]$. Moreover set $P=\sum_{j=1}^N Q_j^2$.
    If $u\in\Distr[\Omega]{n}$ is a solution of $Pu=0$ which vanishes on
    an open subset of $\Omega$ then $u\equiv 0$ in $\Omega$.
\end{Thm}

\section{Applications to CR theory}\label{Section4}
\subsection{Ultradifferentiable CR manifolds}
We may assume in this section, if not otherwise stated,
that $\sheaf{}$ is a weakly normal ultradifferentiable sheaf. 
Moreover, we recall that, if $E$ is a real vector bundle the complexified bundle $\C E=\C\otimes E$,
with fiber $\C\otimes E_p$, is a complex vector bundle.
\begin{Def}
    An involutive, or formally integrable structure $(M,\crb)$ of class $\sheaf{}$ consists of a manifold $M$
    and a subbundle $\crb\subseteq\C TM$, such that the Lie bracket of two sections of $\crb$
    is again a section of $\crb$. (This condition is often informally written as $[\crb,\crb]\subseteq\crb$.)
\end{Def}
If $(M,\crb)$ is an involutive structure then we set $T^\prime=\crb^{\bot}=\Set{\xi\in\C T^\ast M\given \langle \xi, L\rangle =0\;\fa L\in \C TM}\subseteq\C T^\ast M$ and $T^0=T^\prime\cap T^\ast M$
is the characteristic set. We say that a distribution $u\in\Distr[M]{}$ is 
a solution of the structure $(M,\crb)$ if $Lu=0$ for all $L\in \sheaf[M]{\crb}$.
Thence $\WFR u\subseteq T^0\cap S^\ast M$ by \ref{ScalarUltraEll}.

An involutive structure $(M,\crb)$ is called (cf.~\cite[Section I.8]{BCHbook})
\begin{itemize}[label=$\triangleright$]
    \item elliptic if $T^0=\emptyset$.
    \item complex if $\crb\oplus\overline{\crb}=\C T^\ast M$.
    \item CR if $\crb\cap\overline{\crb}=\{0\}$.
    \item essentially real if $\crb=\overline{\crb}$.
\end{itemize}
It is easy to see that a complex structure is an elliptic CR structure, whereas essentially real structures are generated by real vector fields, see \cite{BCHbook}.

We are mainly interest in CR structures. Observe that if $(M,\crb)$ is a CR structure
then $T^0\subseteq T^\ast M$ is actually a real subbundle, see \cite[Proposition I.8.4]{BCHbook}.
Given a CR structure $(M,\crb)$ of class $\sheaf{}$ we denote $\dim M=N=n+m$, $\dim_{\C}\crb=n$ and $\dim T^\prime=m$. Observe that $n\leq m$.

Following \cite[pp.~17--18]{BCHbook}, cf.~\cite{LamelMirRegSections} also \cite{LamelBraun2025}, we can
find local coordinates $(x_1,\dotsc,x_n,y_1,\dotsc,y_n,s_1,\dotsc,s_d)$ defined near 
a point $p\in M$ such that $p=0$ and there are one-forms $\omega_1,\dotsc,\omega_n,\theta_1,\dotsc,\theta_d$ of class $\sheaf{}$ near $p$ such that
\begin{align*}
    \omega_j\vert_p&=dz_j\vert_p=\bigl(dx_j+idy_j\bigr)\vert_p, & j&=1,\dotsc,n\\
    \theta_k\vert_p&=ds_k\vert_p & k&=1,\dotsc,d
\end{align*}
and $\omega_1,\dotsc,\omega_n,\theta_1,\dotsc,\theta_d$ span $T^\prime$ in a neighborhood of $p$ whereas $\theta_1,\dotsc,\theta_d$ generate $T^0$ near $p$.

Moreover, there are functions $a_{jk}$ and $b_{jk}$ of class $\sheaf{}$ near $p$
vanishing at $p$ such the vector fields
\begin{equation*}
    L_j=\frac{\partial }{\partial\bar{z}_j}+\sum_{\ell=1}^n a_{j\ell}\frac{\partial}{\partial z_\ell}+\sum_{k=1}^d b_{jk}\frac{\partial}{\partial s_k},\qquad j=1,\dotsc,n,
\end{equation*}
generate $\sheaf[U]{\crb}$ in some neighborhood $U$ of $p$,
i.e.~$L_1(q),\dotsc , L_n(q)$  is  a basis for $\crb_q$, $q\in U$.
Here we used the convention
\begin{align*}
    \frac{\partial}{\partial z_j}&=\frac{1}{2}\left(\frac{\partial}{\partial x_j}-i\frac{\partial}{\partial y_j}\right)\\
    \frac{\partial}{\partial \bar{z}_j}&=\frac{1}{2}\left(\frac{\partial}{\partial x_j}+i\frac{\partial}{\partial y_j}\right).
\end{align*}

\begin{Rem}[Embedded CR manifolds]
Suppose that $M\subseteq\C[N]$ is a real manifold and note that $T_pM$ can be considered to be a real subspace of $\C^N$, if we identify $\C[N]$ with the underlying real vector
space $\R[N]$ via the map $Z=(x_j+iy_j)_{1\leq j\leq N}\mapsto (x,y)=((x_j),(y_j))$.
Given the complex structure operator $J:\C^N\rightarrow \C^N$ given by $J(Z)=\overline{Z}$,
the complex tangential space $T^c_pM= T_pM\cap JT_p M$ is the maximal complex subspace
of $T_p M$. We say that $M$ is a CR submanifold if $p\mapsto \dim_{\C} T_p^cM$
is constant. Then the CR dimension of $M$ is defined to be $\dim_{\C} T_p^c M$.
If $Z=(Z_1,\dotsc,Z_N)$ denote the (complex) coordinates in $\C^N$ then we observe that
the complexified tangent space $\C T_p\C[N]\cong\C[2N]$ can be split
in two subspaces $T^{1,0}_p$, $T^{0,1}_p$ of (complex) dimension $N$ such that
$\C T_p\C[N]=T^{1,0}_p\oplus T^{0,1}_p$ and
\begin{equation*}
    T^{1,0}_p= \spanc\Set*{\frac{\partial}{\partial Z_j}\Big\vert_p\given j=1,\dotsc, N},\quad T^{0,1}_p=\spanc\Set*{\frac{\partial}{\partial \overline{Z}_j}\Big\vert_p\given j=1,\dotsc,N}.
\end{equation*}
If we set $\crb_p=\C T_p M\cap T^{0,1}_p$ then $\real \crb_p= T_p^c M$ for all $p\in M$. 

If $M\subseteq\C[N]$ is a CR submanifold then we consider the vector bundle 
$\crb=\bigsqcup_{p\in M}\crb_p\subseteq \C TM$. It is easy to see that
$(M,\crb)$ is a CR structure in the sense above. For a far more comprehensive
presentation see \cite{BERbook}, cf.~also \cite{BCHbook}.
We refer to $(M,\crb)$ which might not to be realized as CR submanifolds as
abstract CR manifolds.

In the case of CR submanifolds $M\subseteq\C[N]$ of class $\sheaf{}$, where 
$\sheaf{}$ is a regular sheaf, we observe that we can extend
most of the results of \cite[Section 4]{FuerdoesCR} for the Denjoy-Carleman category.
We may just summarize these statements here.
A CR submanifold $M\subseteq\C[N]$ of class $\sheaf{}$ is generic if
near each $p\in M$ there is a defining function $\rho=(\rho_1,\dotsc,\rho_d): U\rightarrow \R[d]$ of class $\sheaf{}$ in a neighborhood $U\subseteq\C[N]$ of $p$
such that the complex differentials
\begin{equation*}
    \partial \rho_k=\sum_{j=1}^N\frac{\partial \rho_k}{\partial Z_j}dZ_j,\qquad k=1,\dotsc,d,
\end{equation*}
are $\C$ linearly independent near $p$, cf.~\cite[Definition 1.3.4]{BERbook}.
\begin{Prop}
    Let $M\subseteq\C[N]$ be a generic submanifold of class $\sheaf{}$ which is of CR dimension $n$ and of (real) codimension $d$. Then $N=n+d$.
    If $p\in M$ then there are holomorphic coordinates $(z,w)\in\C[n+d]$ defined near $p$
    which vanish at $p$ and a function $\varphi: U\times V\rightarrow \R[d]$
    of class $\sheaf{}$ defined in a neighborhood $U\times V\subseteq\R[2n]\times\R[d]$ of the origin such that $\varphi(0)=0$,
    $d\varphi(0)=0$ such that near $p$ the manifold $M$ is given by
    \begin{equation*}
        \imag w=\varphi(\real z,\imag z,\real w).
    \end{equation*}
\end{Prop}
A CR submanifold $M\subseteq\C[N]$ is minimal at $p\in M$ if there is no
submanifold $S\subseteq M$ through $p$ such that $T_q^c M\subseteq T_q S$ for all $q\in S$ and $\dim_{\R} S<\dim_{\R} M$. 
On the other hand $M$ is of finite type at $p\in M$ if there are vector fields
$X_1,\dotsc, X_k\in \sheaf[M]{T^c M}$ such that the Lie algebra generated
by the $X_1,\dotsc, X_k$ evaluated at $p$ is isomorphic to $T_p M$.
Theorem \ref{NaganoThm} is essential to prove that in the quasianalytic category both
notions coincide:
\begin{Thm}
    Let $\sheaf{}$ be a quasianalytic regular sheaf, $M\subseteq\C[N]$ be a
    CR submanifold of class $\sheaf{}$ and $p\in M$.
    Then $M$ is minimal at $p$ if and only if $M$ is of finite type at $p$.
\end{Thm}
\end{Rem}

\subsection{Regularity of Sections of CR bundles}
\begin{Def}
    Let $(M,\crb)$ be a CR structure of class $\sheaf{}$
    \begin{enumerate}[label=(\alph*),itemsep=0.6ex]
        \item A CR bundle of class $\sheaf{}$ over $M$ is
        a vector bundle $E$ of class $\sheaf{}$ equipped with
    a partial connection $D$, defined on $\crb$.
    More precisely, for each open set $U\subseteq M$ we have that 
    \begin{enumerate}[label=(\roman*),itemsep=0.5ex, leftmargin=3em,topsep=0.4ex]
        \item $D_{aL+K}\omega=aD_L\omega+D_K\omega$ for all $a\in\sheaf[U]{}$,
        $L,K\in\sheaf[U]{\crb}$ and $\omega\in\sheaf[U]{E}$,
        \item $D_L(\omega+\sigma)$ for all $L\in\sheaf[U]{\crb}$ and $\omega,\sigma\in\sheaf[U]{E}$,
        \item $D_L(a\omega)=(La)\omega+aD_L\omega$ for all $a\in\mathcal{C}^1(U)$,
        $L\in\sheaf[U]{\crb}$ and $\omega\in\sheaf[U]{E}$,
        \item $[D_K,D_L]=D_{[K,L]}$ for all $L,K\in\sheaf[U]{\crb}$.
    \end{enumerate}
\item    A (generalized) CR section of $E$ is a section $u\in\Distr[U]{E}$ 
    such that $D_Lu=0$ for all $L\in\sheaf[U]{\crb}$.
%\item A CR module $\module$ is a subsheaf of modules (over the structure sheaf
%$\sheaf{M}$) of the sheaf of modules $\sheaf{E}$, $E$ being a CR bundle over $M$, which is CR closed, i.e.~for every open set $U\subseteq M$, all sections $\omega\in \module[U]$ and every CR vector field $\sheaf[U]{\crb}$
it holds that $D_L\omega\in \module[U]$.
\item More generally, if we consider sheaves of submodules $\Omega$ of $\Distr[]{E}$ (which is
a sheaf of modules over the sheafs of rings $\sheaf{M}$) then we say that $\Omega\subseteq\Distr[]{E}$ is a CR module if it is CR closed.
    \end{enumerate}
\end{Def}
%Generally, given a subset $\Omega$ of $\Distr[U]{E}$ we may write
%\begin{equation*}
%    \left\langle \Omega\right\rangle_{\sheaf{}}
%\end{equation*}
%for the $\sheaf{M}(U)$-module generated by $\Omega$.

From now on $E$ is always a CR bundle over the abstract CR manifold $M$, both of class $\sheaf{}$.
%For an arbitrary submodule $\module\subseteq\sheaf{E}$ we denote by $\hat{\module}$
%the CR hull of $\module$, i.e.~the largest CR module which contains $\module$.
%According to \cite[Lemma 5.3]{LamelBraun2025}, whose proof holds also in the 
%$\sheaf{}$-category, we obtain that there is an increasing sequence of
%submodules given by $\mathcal{N}_0=\module$ and, for $j\in\N$ and $U\subseteq M$
%open,
%\begin{equation*}
%    \mathcal{N}_j(U)=\Set*{\sum_{k}a_k\omega_k+\sum_{\ell} b_k D_{L_k}\theta_k\given
 %   a_k,b_\ell\in\sheaf[U]{M},\; L_\ell\in\sheaf[U]{\crb}}
%\end{equation*}
%such that $\hat{\module}=\bigcup \module_j$
Suppose that $\Omega\subseteq\Distr{E}$ is a CR module.
Following \cite{LamelMirRegSections} and \cite{LamelBraun2025} we set
\begin{equation*}
    \singsupp_{\sheaf{}} \Omega=\overline{\bigcup_{\omega\in\Omega}\singsupp_{\sheaf{}}\omega}
\end{equation*}
and we say that $u\in\Distr[U]{E}$ extends $\sheaf{}$-microlocally if there is
a set $\Gamma\subseteq T^\ast M$ with $\WFR u\subseteq\Gamma$ such that
$\Gamma_p$ is a (non-empty) closed convex cone in $T^\ast_p M\!\setminus\!\{0\}$ for all $p\in M$.
If $u$ is a CR section then we can assume $\Gamma\subseteq T^0M$ by \ref{VectorMicroEll}.

The sheaf of (vector) $\sheaf{}$-multipliers of $\Omega$ is given by 
\begin{equation*}
    \mathcal{S}(\Omega)(U)=\Set*{\eta\in\sheaf[U]{E^\ast}\given \eta(\omega)\in\sheaf[U]{}\;\fa \omega\in\Omega(U)}
\end{equation*}
whereas the sheaf of scalar $\sheaf{}$-multipliers is just
\begin{equation*}
    S(\Omega)(U)=\Set*{\lambda\in\sheaf[U]{}\given \lambda \omega\in\sheaf[U]{}\;\fa\omega\in\Omega}.
\end{equation*}
Note that if either $\mathcal{S}(\module)(U)=\sheaf[U]{E^\ast}$ or 
$S(\module)(U)=\sheaf[U]{M}$ then $\module[U]\subseteq \sheaf[U]{E}$ 
by \emph{\ref{Inversion}}.

The main technical result of \cite{LamelMirRegSections} can be easily generalized
to our setting:
\begin{Prop}\label{MainTechn}
    Let $\module$ be a module over $E$ such that any $\sigma\in\module[U]$
    extends $\sheaf{}$-microlocally. 
    If there is $\lambda\in\sheaf[U]{E^\ast}$ such that $\real \lambda(\sigma)\in
    \sheaf[U]{M}$ for all $\sigma\in\module[U]$ then $\lambda\in\mathcal{S}(\module)$.
\end{Prop}
\begin{proof}
    If $\sigma\in\module[U]$ then $\lambda(\sigma)+\overline{\lambda(\sigma)}=f\in
    \sheaf[U]{M}$ by assumption. For any $p\in M$ we have that $(\WFR \sigma)\vert_p$
    is contained in a closed convex cone $\Gamma_p\subseteq T_p^\ast M\setminus\{0\}$. Using
    \ref{WFInvolution} gives that $\WFR(\overline{\lambda(\sigma)})\vert_p\subseteq 
    -\Gamma_p$. However, applying \ref{BasicWF0} to the assumption we obtain that
    $\WFR (\overline{\lambda(\sigma)}\vert_p\subseteq \WFR(\lambda(\sigma))\vert_p\cup\WFR (f)\vert_p\subseteq\Gamma_p$.
    Since $\Gamma_p$ is convex it follows that $\Gamma_p\cap(-\Gamma_p)=\emptyset$.
    Thus $\lambda(\sigma)\in\sheaf[U]{M}$ and therefore
    $\lambda\in\mathcal{S}(U)$.
\end{proof}
By Theorem \ref{MainTechn} we have that the sheaf of modules defined by
\begin{equation*}
    \mathcal{S}_0(\module)(U)=\Generate*{\lambda\in\sheaf[U]{E^\ast}\given \real \lambda(\sigma)=0\;\fa \sigma\in\module[U]}
\end{equation*}
is contained in $\mathcal{S}(\module)$. 
If $\module$ is a CR module then $\mathcal{S}(\module)$ is a CR bundle with respect
to the CR bundle $E^\ast$ (equipped with the dual connection coming from $E$),
see \cite{LamelMirRegSections} or \cite{LamelBraun2025}.
Therefore the inductively defined chain of sheaves of modules defined by
\begin{equation*}
    \mathcal{S}_j(\module)(U)=\Generate*{\mathcal{S}_{j-1}(\module)(U),D_{L}\mathcal{S}_{j-1}(\module)(U)\given L\in\sheaf[U]{\crb}}
\end{equation*}
is contained in $\mathcal{S}(\module)$ for a CR module $\module$.
We set $\mathcal{S}^\infty(\module)=\bigcup \mathcal{S}_j(\module)\subseteq\mathcal{S}(\module)$.
For each $p\in M$ the stalk $S^\infty(\module)(p)$ can be identified with a subspace of $E_p$. In fact, the mapping $d: M\rightarrow\{0,\dotsc,r\}$
given by $d(p)=r-\dim S^\infty(\module)(p)$ is upper semi-continuous, where $r$
is the rank of $E$.
If we set $\Omega_{\module}^s=(d^{-1}(s))^\circ$ then
\begin{equation*}
    \Omega_{\module}=\bigcup_{j=0}^r\Omega_{\module}^j=\Omega_{\module}^0\sqcup \Sigma_{\module}
\end{equation*}
is an open and dense subset of $M$.
\begin{Thm}\label{MainCRSections}
Let $\module\subseteq\Distr{E}$ be a CR module over $E$ such that any section $\sigma\in \module(U)$
extends microlocally. Then $\Omega^0_{\module}\cap U\subseteq U\setminus\singsupp_{\sheaf{}}\module[U]$ and
$\Sigma_{\module}\cap (\singsupp_{\sheaf{}}\module[U])^\circ$ is dense in 
$(\singsupp_{\sheaf{}}\module[U])$.

Moreover, for any $k=1,\dotsc,r$ there exists a CR subbundle 
$F_k\subseteq E\vert_{\Omega_{\module}^k}$ of rank $k$ such that 
\begin{enumerate}[label=(\roman*),leftmargin=3em]
    \item for any real subbundle $R\subseteq E$ with 
    $\module[\Omega^k]\subseteq\Distr[\Omega^k]{R}$ we have
    that $F_k\subseteq R$
    \item for any $\sigma\in\module[U]$ it follows that the section
    \begin{equation*}
        \tilde{\sigma}: V_k:=U\cap\Omega_{\module}^k\longrightarrow E\vert_{V_k}/F_k
    \end{equation*}
    is of class $\sheaf{}$.
\end{enumerate}
\end{Thm}
\begin{proof}
    It is clear by the arguments above that 
    $\Omega_{\module}^0\cap U\subseteq U\setminus\singsupp_{\sheaf{}} \module(U)$
    and therefore 
    \begin{equation*}
        \Sigma_{\module}\cap(\singsupp_{\sheaf{}}(\module[U])^0
    \end{equation*} 
    has to be dense
    in $\singsupp_{\sheaf{}}\module[U]$ for any open set $U\subseteq M$.

    More generally we have that $\dim \mathcal{S}^\infty(\module)(p)=r-k$
    for any $p\in \Omega_{\module}^k$ for $k=1,\dotsc,r$. Thus we can write
    \begin{equation}\label{Subbundle}
S^\infty(\module)\left(\Omega^k_{\module}\right)=\sheaf[\Omega^k_{\module}]{F_k^{\bot}}
    \end{equation}
    with $F^\bot_k\subseteq E^\ast\vert_{\Omega^k}$ is a CR subbundle of rank $r-k$ since
    $S^\infty$ is a CR module by construction.
    Thus $(F_k^\bot)^\bot\subseteq E^{\ast\ast}\vert_{\Omega^k}=E\vert_{\Omega^k}$.
    Moreover, given some $\lambda\in\sheaf[U]{E^\ast}$ such that 
    $\real \lambda(\omega)=0$ for all $\omega\in\module[U\cap\Omega^k]$
    then $\mathcal{S}_0(\module)(U\cap \Omega^k)\subseteq\mathcal{S}^\infty(\module)(U\cap\Omega^k)$. That means by \eqref{Subbundle}
    that $\lambda(\omega)=0$ for all $\omega\in\sheaf[U\cap\Omega^k]{F_k}$.
    Thence (i) holds. 
    On the other hand, \eqref{Subbundle} is equivalent to 
    $\mathcal{S}^\infty(\module)(U\cap \Omega^k)=\sheaf[U\cap\Omega^k, (E/F_k)^\ast]{}$, which implies (ii).
\end{proof}
Now we can directly extend the main statements of \cite{LamelMirRegSections} 
to the $\sheaf{}$-setting
by, for a given CR section $u\in\Distr[U]{E}$ which extends microlocally, considering the
submodule $\module[U]$ generated by $u$. Then $\module[U]$ is CR closed and
any element extends microlocally. 
We recall also from \cite{LamelMirRegSections} that a real subbundle $R\subseteq E$ is called CR nondegenerate at $p\in M$ if the CR derivatives of sections of $R$
span $E$ at $p$.

Applying Theorem \ref{MainCRSections} gives
\begin{Cor}\label{SectionCor}
    Let $E$ be a CR bundle over the abstract CR manifold $M$, 
    both of class $\sheaf{}$ and $u\in\Distr[U]{E}$ be a CR section. 
    Then the following holds:
    \begin{enumerate}[label=(\roman*),itemsep=0.2ex]
        \item There exists an open dense subset $\Omega_u$ of $(\singsupp_{\sheaf{}} u)^\circ$ 
        such that for each $p\in\Omega_u$ there exists a neighborhood $U$ of $p$
        and a CR bundle $F$ of class $\sheaf{}$ over $F$ with 
        $F\subseteq E\vert_U$ with the following properties:
       For any real subbundle $R\subseteq E$ of class $\sheaf{}$
            such that $\Imag u\vert_U\subseteq R\vert_U$ we have that
            $F\subseteq R\vert_U$.
          Moreover, the section $\tilde{u}: U\rightarrow E\vert_U/F$ is of class
            $\sheaf{}$.
        
        \item If $R\subseteq E$ is a real subbundle which is CR nondegenerate
        at some $p\in U$ such that $\Imag u \subseteq R$ then 
        $u$ is of class $\sheaf{}$ near $p$.
    \end{enumerate}
\end{Cor}

As an application of the theory developed in this subsection we consider
infinitesimal CR automorphisms extending the results of \cite{FLinfAut} to
the $\sheaf{}$-category, see also \cite{FuerdoesCR}.
\begin{Def}
    A $\mathcal{C}^1$ real vector field $X$ defined on an abstract CR manifold is an infinitesimal CR automorphism if $\Fl_X(\,.\,,\tau)$ is a CR automorphism 
    on its domain of definition for $\tau$ small.
\end{Def}
Any real vector field $X\in \Gamma(U,TM)$ induces 
a section $\mathfrak{Y}: U\rightarrow (T^\prime M)^\ast\cong TM/\crb$ by 
setting $\mathfrak{Y}(\omega)=\omega(X)$ for all $\omega\in \sheaf[U]{T^\prime}$. Moreover $\imag \mathfrak{Y}(\theta)=0$
for all $\theta\in\sheaf[U]{T^0M}$.
If $X$ is an infinitesimal CR automorphism then also $\omega(X,L)=0$ for all $\omega\in\sheaf[U]{T^\prime}$ and $L\in\sheaf[U]{\crb}$.

Now note that $T^\prime M$ is a CR bundle with the canonical connection given by $(D_L\omega)(X)=L(\omega(X))-\omega([L,X])$
where $\omega\in\sheaf[U]{T^\prime}$, $X\in\sheaf[U]{\C TM}$
and $L\in\sheaf[U]{\crb}$.
This induces a canonical CR structure on $(T^\prime M)^\ast$ by
setting $(D^\ast_L\mathfrak{Y})(\omega)=L(\mathcal{Y}(\omega))
-\mathfrak{Y}(D_L\omega)$, where $\mathfrak{Y}\in\sheaf[U]{(T^\prime M)^\ast}$, 
$\omega\in\sheaf[U]{T^\prime M}$ and
$L\in\sheaf[U]{\crb}$.
That means that if $X\in\mathcal{C}^1(U,TM)$ is an infinitesimal
CR automorphism then the induced section 
$\mathfrak{Y}\in\mathcal{C}^1(U, (T^\prime M)^\ast)$ 
is CR and $\imag\mathfrak{Y}(\theta)=0$
for all $\theta\in \sheaf[U]{T^0M}$.
This motivates the following definition.
\begin{Def}
    We say that $\mathfrak{Y}\in\Distr[U]{(T^\prime M)^\ast}$
    is a generalized infinitesimal CR automorphism if 
    $\mathfrak{Y}$ is a CR section and $\imag \mathfrak{Y}(\theta)=0$ for all $\theta\in\sheaf[U]{T^0 M}$.
\end{Def}
Moreover, we say that an abstract CR manifold is finitely nondegenerate at $p\in M$ if $(T^\prime)^\ast$ is CR nondegnerate at $p$
with $R=(T^0M)^\ast$ (where in the second case the real dual is considered). We note that this definition of finite nondegeneracy
coincides with the usual one, see e.g.~\cite{FLinfAut}.

Thence Corollary \ref{SectionCor} implies the following result.
\begin{Thm}
    Let $M$ be an abstract CR manifold and
    $\mathfrak{Y}\in\Distr[U]{(T^\prime M)^\ast}$ is 
    a generalized CR automorphism on $U\subseteq M$ which extends 
    $\sheaf{}$-microlocally.
    If $M$ is finitely nondegenerate at $p\in U$ then
    $\mathfrak{Y}$ is of class $\sheaf{}$ near $p$.
\end{Thm}
It is now possible to extend the other results from \cite{FLinfAut}, see
also \cite{FuerdoesCR}, or the statements on the regularity of infinitesmial pertubations of CR immersions given in \cite{LamelMirRegSections}.

\section{Examples of ultradifferential sheaves}\label{Section5}
%\subsection{Ultradifferentiable sheaves given by weight matrices}
\begin{Def}\label{DefWeight}\hspace{3em}
\vspace{0.1ex}
\begin{enumerate}[label=(\roman*)]
    \item A weight sequence is a sequence $\bM=(M_k)_k$ of positive numbers such that $M_0=1$
    and \begin{equation*}
        M_k^2\leq M_{k-1}M_{k+1}\qquad \fa k\in\N.
    \end{equation*}
    \item A weight matrix $\fM$ is a family of weight sequences
    such that for any pair $\bM,\bN\in\fM$ we have either
    $M_k\leq N_k$ for all $k\in\N_0$ or $N_k\leq M_k$ for all $k\in\N_0$.
\end{enumerate}
\end{Def}
Given a weight matrix $\fM$ we can define, following \cite{RainerSchindl14},
two ultradifferentiable sheaves, the Roumieu sheaf $\Rou{\fM}$ and the Beurling sheaf $\Beu{\fM}$.
If $U\subseteq\R[n]$ is open and $f\in\E_n(U)$ then
\begin{itemize}[label=\textbullet]
    \item $f\in\Rou[U]{\fM}$ if for all compact sets $K\subseteq U$ there are
    some $\bM\in\fM$ and constants $C,h>0$ such that
    \begin{equation}\label{DefiningEstimates}
        \sup_{x\in K}\abs*{D^\alpha f(x)}\leq Ch^{\abs{\alpha}}M_{\abs{\alpha}}
\qquad \alpha\in\N_0^n.
    \end{equation}
    \item $f\in\Beu[U]{\fM}$ if for all compact sets $K\subseteq U$, every
    $\bM\in\fM$ and each $h>0$ there is a constant $C>0$ such that
    \eqref{DefiningEstimates} holds.
\end{itemize}
It can be easily seen that both sheaves are isotropic ultradifferentiable sheaves.
\begin{Rem}
    If $\fM=\{\bM\}$ consists of a single weight sequence we obtain
    the classical Denjoy-Carleman classes $\Rou{\bM}$ and $\Beu{\bM}$, 
    see e.g.~\cite{Komatsu}. (Non-)Quasianalyticity of Denjoy-Carleman classes
    is characterized by the famous Denjoy-Carleman theorem:
    Both sheaves $\Rou{\bM}$ and $\Beu{\bM}$ are non-quasianalytic if and only if
    \begin{equation}\label{DCTheorem}
        \sum_{j=0}^\infty\frac{M_j}{M_{j+1}}<\infty.
    \end{equation}
    We say that a weight sequence $\bM$ is non-quasianalytic if 
    \eqref{DCTheorem} holds.
\end{Rem}
Returning to general weight matrices $\fM$ we can
characterize the non-quasianalyticity of the sheaves $\Rou{\fM}$ and $\Beu{\fM}$
according to \cite{Schindl16}:
\begin{itemize}[label=$\triangleright$]
    \item $\Rou{\fM}$ is non-quasianalytic if and only if there is some $\bM\in\fM$ which is non-quasianalytic.
    \item $\Beu{\fM}$ is non-quasianalytic if and only if every $\bM\in\fM$
    is non-quasianalytic.
\end{itemize}

\begin{Def}\label{SemiMatrix}
    Let $\fM$ be a weight matrix.
    \begin{enumerate}[label=(\roman*),itemsep=0.1ex]
        \item $\fM$ is R-semiregular if $\liminf \sqrt[k]{M_k/k!}>0$ for all $\bM\in\fM$ and
        \begin{equation}\label{RDerivClosed}
            \fa \bM \;\ex\bN\; \ex Q>0: \quad M_{k+1}\leq Q^{k+1}N_{k}\qquad \fa k\in\N_0.
        \end{equation}
        \item $\bM$ is B-semiregular if $\lim \sqrt[k]{M_k/k!}=\infty$ for all 
        $\bM\in\fM$ and
        \begin{equation*}
            \fa\bN\in\fM\;\ex\bM\in\fM\;\ex Q>0:\quad M_{k+1}\leq Q^{k+1}N_k\qquad
            \fa k\in\N_0.
        \end{equation*}
    \end{enumerate}
\end{Def}
\begin{Thm}[{\cite{FNRS20}, cf.~Remark \ref{SemiRem}}]
Let $\fM$ be a weight matrix.
\begin{enumerate}[label=(\roman*),itemsep=0.1em]
    \item If $\fM$ is R-semiregular then $\Rou{\fM}$ is semiregular.
    \item If $\fM$ is B-semiregular then $\Beu{\fM}$ is semiregular.
\end{enumerate}
\end{Thm}
In \cite{FNRS20} wavefront sets for $\Rou{\fM}$ and $\Beu{\fM}$ are defined.
\begin{Def}
    Let $\fM$ be a weight matrix, $U\subseteq\R[n]$ open, $u\in\Distr[U]{n}$ and
    $(x_0,\xi_0)\in S^\ast U$.
    \begin{enumerate}[label=(\roman*),itemsep=0.2ex]
        \item $(x_0,\xi_0)\notin\WF_{\{\fM\}}$ if there are a neighborhood $V$ of 
        $x_0$, a bounded sequence $(u_k)_k\subseteq \E^\prime_n(U)$ with
        $u_k\vert_V=u\vert_V$ for all $k\in\N_0$, a neighborhood $\Gamma\subseteq S^{n-1}$ of $\xi_0$, some $\bM\in\fM$ and a constant
        $h>0$ such that 
        \begin{equation}\label{DefWFfM}
            \sup_{\substack{k\in\N_0\\\xi\in\Gamma\\ \lambda>0}}
            \frac{\abs*{\hat{u}_k(\lambda\xi)}}{h^k M_k\lambda^k}<\infty
        \end{equation}
        \item $(x_0,\xi_0)\notin\WF_{(\fM)}$ if there are a neighborhood $V$ of 
        $x_0$, a bounded sequence $(u_k)_k\subseteq \E^\prime_n(U)$ with
        $u_k\vert_V=u\vert_V$ for all $k\in\N_0$ and a neighborhood 
        $\Gamma\subseteq S^{n-1}$ such that \eqref{DefWFfM} holds for all 
        $\bM\in\fM$ and all $h>0$.
    \end{enumerate}
\end{Def}

\begin{Thm}[{\cite{FNRS20} \& \cite{fuerdoes2025}}]
    Let $\fM$ be a weight matrix.
    \begin{enumerate}[label=(\roman*),itemsep=0.4ex]
        \item If $\fM$ is R-semiregular then $\Rou{\fM}$ is microlocal.
        \item If $\fM$ is B-semiregular then $\Beu{\fM}$ is microlocal.
    \end{enumerate}
\end{Thm}

\begin{Thm}[{\cite{RainerSchindl16}}]\leavevmode
 \begin{enumerate}[label=(\roman*),itemsep=0.4ex]
     \item If $\fM$ is a R-semiregular weight matrix then
     $\Rou{\fM}$ is regular if and only if 
     \begin{equation*}
         \fa\bM\in\fM\;\ex\bN\in\fM\;\ex C>0:\quad \left(\frac{M_j}{j!}\right)^{1/j}\leq C\left(\frac{N_k}{k!}\right)^{1/k}\qquad \fa j\leq k.
     \end{equation*}
     \item If $\fM$ is a B-semiregular weight matrix then
     $\Beu{\fM}$ is regular if and only if 
     \begin{equation*}
         \fa\bN\in\fM\;\ex\bM\in\fM\,\ex C>0:\quad \left(\frac{M_j}{j!}\right)^{1/j}\leq C\left(\frac{N_k}{k!}\right)^{1/k}\qquad \fa j\leq k.
     \end{equation*}
 \end{enumerate}
\end{Thm}
We say that a weight sequence $\bM$ is strongly log-convex if
\begin{equation*}
    \left(\frac{M_k}{k!}\right)^2\leq \frac{M_{k-1}}{(k-1)!}\frac{M_{k+1}}{(k+1)!}\qquad k\in\N.
\end{equation*}
It is easy to see that if $\bM$ is strongly log-convex then
the sequence $(M_k/k!)^{1/k}$ is increasing.
\begin{Def}\label{NormalMatrix}
Let $\fM$ be a weight matrix.
\begin{enumerate}[label=(\roman*),itemsep=0.4ex]
    \item $\fM$ is R-normal if $\fM$ is R-semiregular,
    every $\bM\in\fM$ is strongly log-convex and
    \begin{equation*}
        \fa\bM\in\fM\;\ex\bN\in\fM\;\ex Q>0:\quad M_{j+k}\leq Q^{j+k+1}N_jN_k\qquad j,k\in\N_0.
    \end{equation*}
    \item $\fM$ is B-normal if $\fM$ is B-semiregular, every $\bM\in\fM$
    is strongly log-convex and
    \begin{equation*}
        \fa\bN\in\fM\;\ex\bM\in\fM\;\ex Q>0:\quad M_{j+k}\leq Q^{j+k+1}N_jN_k\qquad j,k\in\N_0.
    \end{equation*}
\end{enumerate}
\end{Def}
\begin{Thm}[{\cite{FNRS20}}]
    Let $\fM$ is a R-normal weight matrix then $\Rou{\fM}$ is weakly normal.
    Moreover, if $\Rou{\fM}$ is quasianalytic then \ref{QuasiWF} holds.
\end{Thm}
It is not known if $\Beu{\fM}$ is weakly normal if $\fM$ is B-normal.
However for Denjoy-Carleman classes we have a positive result.
Since the definitions for normality coincide for singleton $\fM=\{\bM\}$, we say then that the weight sequence $\bM$ is normal.
\begin{Thm}[{\cite{FNRS20}}]
    If $\bM$ is a normal weight sequence then both
    $\Rou{\bM}$ and $\Beu{\bM}$ are both weakly normal sheaves.
\end{Thm}

The one question still open is if $\Rou{\fM}$ is in fact a normal sheaf for
normal weight matrices $\fM$.
Indeed, we have that $\Rou{\fM}$ satisfies \ref{UltraWFInv0} if $\fM$ is R-normal
and, moreover, \ref{UltraWFInv0} holds in $\Beu{\fM}$ when $\fM$ is B-normal,
see \cite{FNRS20}.

Thence we need only to verify \ref{VectorMicroEll}. 
\begin{Thm}
    Let $\fM$ be a R-normal weight matrix and $\Omega\subseteq\R[n]$ be an open set. 
    If $P=(P_{jk})_{j,k}$ is a linear differential operator acting on $\Distr[\Omega;{\R[\nu]}]{n}$ with coefficients in $\Rou[\Omega]{\fM}$ then 
    \begin{equation*}
        \WF_{\{\fM\}} u\subseteq\WF_{\{\fM\}}Pu \cup\Char P \qquad \fa u\in\Distr[\Omega;{\C[\nu]}]{n}.
    \end{equation*}
\end{Thm}
\begin{proof}
    We give here only a sketch of the proof, since we need only
    to combine the proofs of \cite[Theorem 7.1]{FNRS20} and 
    of \cite[Theorem 6.1]{Fuerdoes20}.

    Since we can identify closed subsets of $S^\ast \Omega$ with closed conic subsets
    of $T^\ast \Omega\!\setminus\{0\}$ we consider $\WF_{\{\fM\}}$ and $\Char P$
    as subsets of $T^\ast \Omega\setminus\{0\}$. 
    We note also it is enough to show the implication
    \begin{equation*}
        (x_0,\xi_0)\notin\WF_{\{\fM\}}Pu\cup\Char P\Longrightarrow 
        (x_0,\xi_0)\notin\WF_{\{\fM\}}u
    \end{equation*}
    for all $u\in\Distr[\Omega,{\C[\nu]}]{n}$.

    Hence assume that $(x_0,\xi_0)\notin (x_0,\xi_0)\notin\WF_{\{\fM\}}f\cup\Char P$ where $f=Pu$. Then there has to be a compact neighborhood $K$ of $x_0$
    in $\Omega$ and a closed conic neighborhood $V\subseteq\R^n\setminus\{0\}$
    of $\xi_0$
    such that
    \begin{align*}
        \det p(x,\xi)&\neq 0 & (x,\xi)&\in K\times V\\
        K\times V\cap \WF_{\{\fM\}} f_j&=\emptyset & j&=1,\dotsc,\nu,
    \end{align*}
    where $p(x,\xi)$ is the principal symbol of $P$.

    Furthermore, we consider a formal adjoint $Q=P^t$ of $P$ with respect to
    the pairing $\langle f,g\rangle=\sum \int f_jg_j$ defined for 
    $f,g\in\D(\Omega;\C[\nu])$. Then $Q=(Q_{jk})$ where $Q_{jk}=P^t_{kj}$ is the formal adjoint
    of the scalar operator $P_{kj}$.
We also choose a sequence of test functions $\lambda_N\in\D(K)$ satisfying 
$\lambda_N\vert_V\equiv 1$ on a fixed neighborhood $U$ of $x_0$ and
for all $\alpha\in \N_0^n$ there are constants $C_\alpha,h_\alpha>0$ such that
\begin{equation*}
    \sup_{x\in K}\abs*{D^{\alpha+\beta}\lambda(x)}
    \leq C_\alpha h^{\abs{\beta}}_\alpha N^{\abs{N}}\qquad \beta\in\N_0^n,\,\abs{\beta}\leq N.
\end{equation*}
    If $u=(u_1,\dotsc,u_\nu)\in\Distr[\Omega,{\C[\nu]}]{n}$ then
    the sequence $u_N^\tau=\lambda_{2N}u^\tau$ is bounded in $\E^\prime$
    and each of these distributions equals $u^\tau$ in $U$ for every
    $\tau=1,\dotsc,\nu$.
    
    It follows that by \eqref{DefWFfM} it is enough  to show that there are some $\bM\in\fM$ and
    a constant $Q>0$ such
    \begin{equation}\label{Target}
        \sup_{\substack{\xi\in V\\ N\in\N_0}} \frac{\abs*{\xi}^N \abs*{\hat{u}^\tau_N(\xi)}}{Q^NM_N}<\infty \qquad \tau=1,\dotsc,\nu.
    \end{equation}
    As in the proofs of \cite[Theorem 4.1]{AJO10}
    \cite[Theorem 6.1]{Fuerdoes20} or \cite[Theorem 7.1]{FNRS20} we want to adapt the proof
    of \cite[Theorem 8.6.1]{HoeBook1}. In order to do so, set
    $\Lambda_N^\tau=\lambda_N e^\tau\in\D(K;{\C[\nu]})$ where $e^\tau$ is the
    $\tau$-th unit vector in $\C[\nu]$, for $\tau=1,\dotsc,\nu$.
    We want to solve the equation $Qg^\tau=e^{-ix\xi}\Lambda^\tau_{2N}$
    and start by making the ansatz $g^\tau=e^{-ix\xi}B(x,\xi)w^\tau$
    where $B(x,\xi)$ is the inverse matrix of the transpose of $p(x,\xi)$
    which is well-defined for $(x,\xi)\in K\times V$ and homogeneous 
    of degree $-d$ in $\xi$.
    This leads to the following equation for $w^\tau$:
    \begin{equation}\label{LastThm-E1}
        \omega^\tau-R\omega^\tau=\Lambda^N_{2N}.
    \end{equation}
    Here $R=R_1+\dots+R_d$ with $R_j\abs{\xi}^j$ being a (matrix) differential operator of order $j$ with coefficients in $\Rou{\fM}$ and which are
    homogeneous of degree $0$ in $\xi$ for $x\in K$ and $\xi\in V$.
    As noted in \cite[p.~306f]{HoeBook1} a formal solution of 
    \eqref{LastThm-E1} would be $W^\tau=\sum R^k \Lambda_{2N}$ but we 
    are not allowed to consider arbitrary high order derivatives.
    Hence we consider approximate solutions given by
    \begin{equation*}
        W_N^\tau=\sum_{j_1+\dots+j_k\leq N-d}R_{j_1}\dots R_{j_k}\Lambda_{2N}^\tau
    \end{equation*}
    for $\tau=1,\dotsc,\nu$. Then
    \begin{equation*}
        w^\tau_N-Rw^\tau_N=\Lambda_{2N}^\tau-\rho_N^\tau.
    \end{equation*}
    Then \cite[(6-8)]{Fuerdoes20} gives that
    \begin{equation}\label{LastThmE2}
        \hat{u}_N^\tau(\xi)=
        \bigl\langle f,e^{-\,.\,\xi}B(\,.\,,\xi)w^\tau_N(\,.\,,\xi)\bigr\rangle
        +\bigl\langle u, e^{-\,.\,\xi}\rho_N^\tau(\,.\,,\xi)\bigr\rangle
    \end{equation}
    and we continue by estimating the right-hand side of \eqref{LastThmE2}.
    Observe that we can use the deduction of \cite[(7.9)]{FNRS20} 
    (cf.~\cite[(6-9)]{Fuerdoes20}) to show that there are some $\bM\in\fM$
    and constants $C_1,h_1>0$ such that  
    \begin{equation}\label{LastThmE3}
      \abs*{\bigl\langle u, e^{-\,.\,\xi}\rho_N^\tau(\,.\,,\xi)\bigr\rangle}
      \leq C_1h^{N}_1\abs{\xi}^{\mu+d-N}M_N,\qquad \fa N\in\N_0
    \end{equation}
    for $\xi\in V$, $\abs{\xi}>1$. Here $\mu$ denotes the order of $u$ in some neighborhood of $K$.
    On the other hand, a comparison of \cite[pp.~339--341]{Fuerdoes20}
    with \cite[pp.~43-44]{FNRS20} leads us to the existence of
    $\bM^\prime,\bM^{\prime\prime}\in\fM$ and $C_2,h_2>0$ such that
    \begin{equation}\label{LastThmE4}
\abs*{\bigl\langle f,e^{-\,.\,\xi}B(\,.\,,\xi)w^\tau_N(\,.\,,\xi)\bigr\rangle}
\leq C_2 h_2^N M^{\prime\prime}_N\abs*{\xi}^{\mu+n-N} \qquad \fa N
    \end{equation}
    for $\xi\in V$, $\abs{\xi}\geq \sqrt[N]{M^\prime_N}$.
Finally, since the sequence $(u^\tau_N)_N$ is bounded in $\E^{\prime,\mu}$,
the Banach-Steinhaus Theorem implies that  there are constants $C_3,h_3>0$
 and $\widetilde{\bM}\in\fM$ such that
 \begin{equation}\label{LastThmE5}
    \abs*{\xi}^N\abs*{\hat{u}^\tau_N}\leq C_3h_3 \widetilde{M}_{N}\qquad \fa N\in\N_0
\end{equation}
 for $\abs*{\xi}\leq \sqrt[N]{M_N^\prime}$, cf.~\cite[(6-12)]{Fuerdoes20}
 and \cite[(7.15)]{FNRS20}

 According to Definition \ref{DefWeight}(ii) we can assume that
 $\bM=\bM^{\prime\prime}=\widetilde{\bM}$ and thus combining \eqref{LastThmE3},
 \eqref{LastThmE4} and \eqref{LastThmE5} and using \eqref{RDerivClosed}
shows that \eqref{Target} holds.
\end{proof}
Using \cite[Lemma 7.5]{FNRS20} analogous as in the proof of \cite[Theorem 7.4]{FNRS20} we are going to obtain \ref{VectorMicroEll} for $\Beu{\bM}$ with
$\bM$ being a normal weight sequence.
\begin{Cor}
Let $\bM$ be a normal weight sequence. Then both $\Rou{\bM}$ and $\Beu{\bM}$
are normal ultradifferentiable classes.
\end{Cor}
\begin{Rem}
    The advantage of working with classes of weight matrices is that they provide
    a unified framework for the two most well-known families of ultradifferentiable classes, Denjoy-Carleman classes defined by weight sequences
    and Braun-Meise-Taylor classes given by weight functions.
    Note that in general these classes are distinct of each other, see \cite{BMM07}.
    
    For the definition of weight functions and Braun-Meise-Taylor classes
    we may refer, e.g., to \cite{RainerSchindl14} and summarize here only the properties of Braun-Meise-Taylor classes.
    \begin{Thm}[cf.~{\cite{AJO10}, \cite{FNRS20}, \cite{fuerdoes2025}}]
Let $\omega$ be a weight function in the sense of \cite{BMT90}.
\begin{enumerate}[label=(\roman*)]
    \item If $\omega(t)=O(t)$ when $t\rightarrow \infty$ then $\Rou{\omega}$
    is semiregular and microlocal. If $\omega$ is additionally equivalent to a subadditive
    weight function then $\Rou{\omega}$ is normal.
    \item If $\omega(t)=o(t)$, $t\rightarrow\infty$ then $\Beu{\omega}$ is
    a semiregular and microlocal sheaf. If $\omega$ is also equivalent to 
       a subadditive weight function then $\Beu{\omega}$ is normal.
\end{enumerate}
    \end{Thm}
    
\end{Rem}
\bibliographystyle{plainurl}
\bibliography{References}

@article{AJO10,
 author = {Albanese, A. A. and Jornet, D. and Oliaro, A.},
 title = {Quasianalytic wave front sets for solutions of linear partial differential operators},
 fjournal = {Integral Equations and Operator Theory},
 journal = {Integral Equations Oper. Theory},
 issn = {0378-620X},
 volume = {66},
 number = {2},
 pages = {153--181},
 year = {2010},
 language = {English},
 doi = {10.1007/s00020-010-1742-6},
 zbMATH = {5774247},
 Zbl = {1208.46037}
}

@book{BERbook,
 author = {Baouendi, M. Salah and Ebenfelt, Peter and Rothschild, Linda Preiss},
 title = {Real submanifolds in complex space and their mappings},
 fseries = {Princeton Mathematical Series},
 series = {Princeton Math. Ser.},
 volume = {47},
 isbn = {0-691-00498-6},
 year = {1999},
 publisher = {Princeton, NJ: Princeton University Press},
 language = {English},
 zbMATH = {1249530},
 Zbl = {0944.32040}
}

@book {BCHbook,
    AUTHOR = {Berhanu, Shiferaw and Cordaro, Paulo D. and Hounie, Jorge},
     TITLE = {An introduction to involutive structures},
    SERIES = {New Mathematical Monographs},
    VOLUME = {6},
 PUBLISHER = {Cambridge University Press, Cambridge},
      YEAR = {2008},
     PAGES = {xii+392},
      ISBN = {978-0-521-87857-9},
   MRCLASS = {32V05 (35-02 35D10 35F05 35N10 58J10)},
  MRNUMBER = {2397326},
MRREVIEWER = {Alberto\ Parmeggiani},
       DOI = {10.1017/CBO9780511543067},
       URL = {https://doi.org/10.1017/CBO9780511543067},
}

@article{BierstoneMilman,
 author = {Bierstone, Edward and Milman, Pierre D.},
 title = {Resolution of singularities in {Denjoy}-{Carleman} classes},
 fjournal = {Selecta Mathematica. New Series},
 journal = {Sel. Math., New Ser.},
 issn = {1022-1824},
 volume = {10},
 number = {1},
 pages = {1--28},
 year = {2004},
 language = {English},
 doi = {10.1007/s00029-004-0327-0},
 zbMATH = {2103782},
 Zbl = {1078.14087}
}

@misc{Bony76,
    author = {Bony, J. M.},
    title = {Extensions du th{\'e}or{\`e}me de {Holmgren}},
    year = {1976},
    language = {French},
    howpublished = {S{\'e}min. {Goulaouic}-{Schwartz} 1975-1976, {\'E}quat. d{\'e}riv. part. {Anal}. fonct., {Expos{\'e}} {No}.{XVII}, 12 p. (1976).},
    url = {https://eudml.org/doc/111657},
    zbMATH = {3525577},
    Zbl = {0336.35003}
}

@article{Bony69,
    author = {Bony, Jean Michel},
    title = {An extension of {Holmgren}'s theorem on the uniqueness of the {Cauchy} problem},
    fjournal = {Comptes Rendus Hebdomadaires des S{\'e}ances de l'Acad{\'e}mie des Sciences, S{\'e}rie A},
    journal = {C. R. Acad. Sci., Paris, S{\'e}r. A},
    issn = {0366-6034},
    volume = {268},
    pages = {1103--1106},
    year = {1969},
    language = {French},
    zbMATH = {3274933},
    Zbl = {0172.38001}
}

@article{BMT90,
 author = {Braun, R. W. and Meise, R. and Taylor, B. A.},
 title = {Ultradifferentiable functions and {Fourier} analysis},
 fjournal = {Results in Mathematics},
 journal = {Result. Math.},
 issn = {1422-6383},
 volume = {17},
 number = {3-4},
 pages = {206--237},
 year = {1990},
 language = {English},
 doi = {10.1007/BF03322459},
 zbMATH = {5550},
 Zbl = {0735.46022}
}

@article{BMM07,
 author = {Bonet, Jos{\'e} and Meise, Reinhold and Melikhov, Sergej N.},
 title = {A comparison of two different ways to define classes of ultradifferentiable functions},
 fjournal = {Bulletin of the Belgian Mathematical Society - Simon Stevin},
 journal = {Bull. Belg. Math. Soc. - Simon Stevin},
 issn = {1370-1444},
 volume = {14},
 number = {3},
 pages = {425--444},
 year = {2007},
 language = {English},
 zbMATH = {5231820},
 Zbl = {1165.26015}
}

@misc{LamelBraun2025,
      title={Regularity of infinitesimal automorphisms of involutive structures}, 
      author={Bernhard Lamel and Nicholas Braun Rodrigues},
      year={2025},
      eprint={2506.24105},
      archivePrefix={arXiv},
      primaryClass={math.CV},
      url={https://arxiv.org/abs/2506.24105}, 
}

@article{Chinni23,
 author = {Chinni, G.},
 title = {On the partial and microlocal regularity for generalized {M{\'e}tivier} operators},
 fjournal = {Journal of Pseudo-Differential Operators and Applications},
 journal = {J. Pseudo-Differ. Oper. Appl.},
 issn = {1662-9981},
 volume = {14},
 number = {3},
 pages = {26},
 note = {Id/No 38},
 year = {2023},
 language = {English},
 doi = {10.1007/s11868-023-00534-6},
 url = {hdl.handle.net/11585/928379},
 zbMATH = {7707782},
 Zbl = {1519.35045}
}

@PhdThesis{FuerdoesThesis,
  author = {Stefan F\"urd\"os},
  school = {Faculty of Mathematics, University of Vienna},
  title  = {Ultradifferentiable CR manifolds},
  year   = {2017},
  doi    = {https://doi.org/10.25365/thesis.49161},
  url    = {https://phaidra.univie.ac.at/o:1338656},
}

@article{fuerdoes2025,
      title={Hypoellipticity of analytic differential operators in general ultradifferentiable classes}, 
      author={Stefan F{\"u}rd{\"o}s},
      journal={J. Pseudo-Differ. Oper. Appl.},
      year={2026},
      eprint={2511.13579},
      archivePrefix={arXiv},
      primaryClass={math.AP},
      url={https://arxiv.org/abs/2511.13579}, 
      note={To be published}, 
      doi={10.1007/s11868-026-00803-0}
}

@article {FuerdoesCR,
    AUTHOR = {F\"urd\"os, Stefan},
     TITLE = {Ultradifferentiable {CR} manifolds},
   JOURNAL = {J. Geom. Anal.},
  FJOURNAL = {Journal of Geometric Analysis},
    VOLUME = {30},
      YEAR = {2020},
    NUMBER = {3},
     PAGES = {3064--3098},
      ISSN = {1050-6926,1559-002X},
   MRCLASS = {32V05 (26E10 32V99 35A18)},
  MRNUMBER = {4105146},
MRREVIEWER = {Song\ Ying\ Li},
       DOI = {10.1007/s12220-019-00191-6},
       URL = {https://doi.org/10.1007/s12220-019-00191-6},
}

@article{Fuerdoes20,
 author = {F{\"u}rd{\"o}s, Stefan},
 title = {Geometric microlocal analysis in {Denjoy}-{Carleman} classes},
 fjournal = {Pacific Journal of Mathematics},
 journal = {Pac. J. Math.},
 issn = {1945-5844},
 volume = {307},
 number = {2},
 pages = {303--351},
 year = {2020},
 language = {English},
 doi = {10.2140/pjm.2020.307.303},
 zbMATH = {7264015},
 Zbl = {1451.35007}
}

@article {FLinfAut,
    AUTHOR = {F\"urd\"os, Stefan and Lamel, Bernhard},
     TITLE = {Regularity of infinitesimal {CR} automorphisms},
   JOURNAL = {Internat. J. Math.},
  FJOURNAL = {International Journal of Mathematics},
    VOLUME = {27},
      YEAR = {2016},
    NUMBER = {14},
     PAGES = {1650112, 25},
      ISSN = {0129-167X,1793-6519},
   MRCLASS = {32V05 (35A18 35N15)},
  MRNUMBER = {3593674},
MRREVIEWER = {Gerd\ Schmalz},
       DOI = {10.1142/S0129167X16501123},
       URL = {https://doi.org/10.1142/S0129167X16501123},
}

@article{FNRS20,
 author = {F{\"u}rd{\"o}s, Stefan and Nenning, David Nicolas and Rainer, Armin and Schindl, Gerhard},
 title = {Almost analytic extensions of ultradifferentiable functions with applications to microlocal analysis},
 fjournal = {Journal of Mathematical Analysis and Applications},
 journal = {J. Math. Anal. Appl.},
 issn = {0022-247X},
 volume = {481},
 number = {1},
 pages = {51},
 note = {Id/No 123451},
 year = {2020},
 language = {English},
 doi = {10.1016/j.jmaa.2019.123451},
 zbMATH = {7113638},
 Zbl = {1427.32009}
}

@article{Hoe67,
 author = {H{\"o}rmander, Lars},
 title = {Hypoelliptic second order differential equations},
 fjournal = {Acta Mathematica},
 journal = {Acta Math.},
 issn = {0001-5962},
 volume = {119},
 pages = {147--171},
 year = {1967},
 language = {English},
 doi = {10.1007/BF02392081},
 zbMATH = {3250869},
 Zbl = {0156.10701}
}

@book {HoeBook1,
    AUTHOR = {H\"ormander, Lars},
     TITLE = {The analysis of linear partial differential operators. {I}},
    SERIES = {Classics in Mathematics},
      NOTE = {Distribution theory and Fourier analysis,
              Reprint of the second (1990) edition [Springer, Berlin;
              MR1065993 (91m:35001a)]},
 PUBLISHER = {Springer-Verlag, Berlin},
      YEAR = {2003},
     PAGES = {x+440},
      ISBN = {3-540-00662-1},
   MRCLASS = {35-02},
  MRNUMBER = {1996773},
       DOI = {10.1007/978-3-642-61497-2},
       URL = {https://doi.org/10.1007/978-3-642-61497-2},
}

@article{Komatsu,
 author = {Komatsu, Hikosaburo},
 title = {Ultradistributions. {I}: {Structure} theorems and a characterization},
 fjournal = {Journal of the Faculty of Science. Section I A},
 journal = {J. Fac. Sci., Univ. Tokyo, Sect. I A},
 issn = {0040-8980},
 volume = {20},
 pages = {25--105},
 year = {1973},
 language = {English},
 zbMATH = {3408337},
 Zbl = {0258.46039}
}

@article{LiessRodino84,
 author = {Liess, Otto and Rodino, Luigi},
 title = {Inhomogeneous {Gevrey} classes and related pseudodifferential operators},
 fjournal = {Bollettino della Unione Matem{\`a}tica Italiana. Serie VI. C. Analisi Funzionale e Applicazioni},
 journal = {Boll. Unione Mat. Ital., VI. Ser., C, Anal. Funz. Appl.},
 volume = {3},
 pages = {233--323},
 year = {1984},
 language = {English},
 zbMATH = {3888250},
 Zbl = {0557.35131}
}

@article {LamelMirRegSections,
    AUTHOR = {Lamel, Bernhard and Mir, Nordine},
     TITLE = {Regularity of sections of {CR} vector bundles},
   JOURNAL = {Proc. Amer. Math. Soc. Ser. B},
  FJOURNAL = {Proceedings of the American Mathematical Society. Series B},
    VOLUME = {11},
      YEAR = {2024},
     PAGES = {15--22},
      ISSN = {2330-1511},
   MRCLASS = {32V05 (32V20)},
  MRNUMBER = {4713116},
MRREVIEWER = {Gerd\ Schmalz},
       DOI = {10.1090/bproc/149},
       URL = {https://doi.org/10.1090/bproc/149},
}

@article{Matsumoto,
 author = {Matsumoto, Waichiro},
 title = {Theory of pseudo-differential operators of ultradifferentiable class},
 fjournal = {Journal of Mathematics of Kyoto University},
 journal = {J. Math. Kyoto Univ.},
 issn = {0023-608X},
 volume = {27},
 pages = {453--500},
 year = {1987},
 language = {English},
 doi = {10.1215/kjm/1250520659},
 zbMATH = {4061820},
 Zbl = {0651.35088}
}

@article{NaganoPaper,
 author = {Nagano, T.},
 title = {Linear differential systems with singularities and an application to transitive {Lie} algebras},
 fjournal = {Journal of the Mathematical Society of Japan},
 journal = {J. Math. Soc. Japan},
 issn = {0025-5645},
 volume = {18},
 pages = {398--404},
 year = {1966},
 language = {English},
 doi = {10.2969/jmsj/01840398},
 zbMATH = {3237708},
 Zbl = {0147.23502}
}

@article{Petrowsky1939,
  author   = {Petrowsky, I. G.},
  journal  = {Rec. Math. Moscou, n. Ser.},
  title    = {Sur l'analyticite des solutions des syst{\`e}mes d'{\'e}quations diff{\'e}rentielles},
  year     = {1939},
  pages    = {3--68},
  volume   = {5},
  fjournal = {Recueil Math{\'e}matique. Nouvelle S{\'e}rie},
  language = {French},
  zbl      = {0022.22601},
  zbmath   = {3036101},
}

@article{RainerSchindl14,
 author = {Rainer, Armin and Schindl, Gerhard},
 title = {Composition in ultradifferentiable classes},
 fjournal = {Studia Mathematica},
 journal = {Stud. Math.},
 issn = {0039-3223},
 volume = {224},
 number = {2},
 pages = {97--131},
 year = {2014},
 language = {English},
 doi = {10.4064/sm224-2-1},
 zbMATH = {6373523},
 Zbl = {1318.26053}
}

@article{RainerSchindl16,
 author = {Rainer, Armin and Schindl, Gerhard},
 title = {Equivalence of stability properties for ultradifferentiable function classes},
 fjournal = {Revista de la Real Academia de Ciencias Exactas, F{\'{\i}}sicas y Naturales. Serie A: Matem{\'a}ticas. RACSAM},
 journal = {Rev. R. Acad. Cienc. Exactas F{\'{\i}}s. Nat., Ser. A Mat., RACSAM},
 issn = {1578-7303},
 volume = {110},
 number = {1},
 pages = {17--32},
 year = {2016},
 language = {English},
 doi = {10.1007/s13398-014-0216-0},
 zbMATH = {6558160},
 Zbl = {1339.26078}
}

@book {RodinoBook,
    AUTHOR = {Rodino, Luigi},
     TITLE = {Linear partial differential operators in {G}evrey spaces},
 PUBLISHER = {World Scientific Publishing Co., Inc., River Edge, NJ},
      YEAR = {1993},
     PAGES = {x+251},
      ISBN = {981-02-0845-6},
   MRCLASS = {35-02 (35B27 35D05 35H05 35S05 46Fxx)},
  MRNUMBER = {1249275},
MRREVIEWER = {Todor\ V.\ Gramchev},
       DOI = {10.1142/9789814360036},
       URL = {https://doi.org/10.1142/9789814360036},
}

@book{RudinBook,
 author = {Rudin, Walter},
 title = {Real and complex analysis.},
 edition = {3rd ed.},
 isbn = {0-07-054234-1},
 year = {1987},
 publisher = {New York, NY: McGraw-Hill},
 language = {English},
 zbMATH = {1022658},
 Zbl = {0925.00005}
}

@article{Schapira68,
 author = {Schapira, P.},
 title = {Sur les ultradistributions},
 fjournal = {Annales Scientifiques de l'{\'E}cole Normale Sup{\'e}rieure. Quatri{\`e}me S{\'e}rie},
 journal = {Ann. Sci. {\'E}c. Norm. Sup{\'e}r. (4)},
 issn = {0012-9593},
 volume = {1},
 number = {3},
 pages = {395--415},
 year = {1968},
 language = {French},
 doi = {10.24033/asens.1167},
 url = {https://eudml.org/doc/81836},
 zbMATH = {3262719},
 Zbl = {0164.15401}
}

@article{Schindl16,
 author = {Schindl, Gerhard},
 title = {Characterization of ultradifferentiable test functions defined by weight matrices in terms of their {Fourier} transform},
 fjournal = {Note di Matematica},
 journal = {Note Mat.},
 issn = {1123-2536},
 volume = {36},
 number = {2},
 pages = {1--35},
 year = {2016},
 language = {English},
 doi = {10.1285/i15900932v36n2p1},
 zbMATH = {6839676},
 Zbl = {1382.26027}
}

@article{Sussmann,
 author = {Sussmann, Hector J.},
 title = {Orbits of families of vector fields and integrability of distributions},
 fjournal = {Transactions of the American Mathematical Society},
 journal = {Trans. Am. Math. Soc.},
 issn = {0002-9947},
 volume = {180},
 pages = {171--188},
 year = {1973},
 language = {English},
 doi = {10.2307/1996660},
 zbMATH = {3431564},
 Zbl = {0274.58002}
}

@article{Treves71,
 author = {Tr{\`e}ves, Fran{\c{c}}ois},
 title = {Analytic-hypoelliptic partial differential equations of principal type},
 fjournal = {Communications on Pure and Applied Mathematics},
 journal = {Commun. Pure Appl. Math.},
 issn = {0010-3640},
 volume = {24},
 pages = {537--570},
 year = {1971},
 language = {English},
 doi = {10.1002/cpa.3160240407},
 zbMATH = {3352125},
 Zbl = {0222.35014}
}

@book {TrevesBook80a,
    AUTHOR = {Tr\`eves, Fran{\c{c}}ois},
     TITLE = {Introduction to pseudodifferential and {F}ourier integral
              operators. {V}ol. 1},
    SERIES = {The University Series in Mathematics},
      NOTE = {Pseudodifferential operators},
 PUBLISHER = {Plenum Press, New York-London},
      YEAR = {1980},
     PAGES = {xxvii+299+xi},
      ISBN = {0-306-40403-6},
   MRCLASS = {35Sxx (47G05 58G15)},
  MRNUMBER = {597144},
MRREVIEWER = {Vesselin\ M.\ Petkov},
}

@book {TrevesBook80b,
    AUTHOR = {Tr\`eves, Fran{\c{c}}ois},
     TITLE = {Introduction to pseudodifferential and {F}ourier integral
              operators. {V}ol. 2},
    SERIES = {The University Series in Mathematics},
      NOTE = {Fourier integral operators},
 PUBLISHER = {Plenum Press, New York-London},
      YEAR = {1980},
     PAGES = {xiv+301--649+xi},
      ISBN = {0-306-40404-4},
   MRCLASS = {58G15 (35Sxx 47G05)},
  MRNUMBER = {597145},
MRREVIEWER = {Vesselin\ M.\ Petkov},
}

@book {TrevesBook25,
    AUTHOR = {Treves, Fran{\c{c}}ois},
     TITLE = {Analytic partial differential equations},
    SERIES = {Grundlehren der mathematischen Wissenschaften [Fundamental
              Principles of Mathematical Sciences]},
    VOLUME = {359},
 PUBLISHER = {Springer, Cham},
      YEAR = {[2022] \copyright 2022},
     PAGES = {xiii+1228},
      ISBN = {978-3-030-94054-6; 978-3-030-94055-3},
   MRCLASS = {35-02 (35A01 35A09 35A10 46-02)},
  MRNUMBER = {4436039},
       DOI = {10.1007/978-3-030-94055-3},
       URL = {https://doi.org/10.1007/978-3-030-94055-3},
}

@article{Zachm69,
 author = {Zachmanoglou, E. C.},
 title = {An application of {Holmgren}'s theorem and convexity with respect to differential operators with flat characteristic cones},
 fjournal = {Transactions of the American Mathematical Society},
 journal = {Trans. Am. Math. Soc.},
 issn = {0002-9947},
 volume = {140},
 pages = {109--115},
 year = {1969},
 language = {English},
 doi = {10.2307/1995126},
 zbMATH = {3286498},
 Zbl = {0179.19304}
}

@article{Zachm72,
 author = {Zachmanoglou, E. C.},
 title = {Foliations and the propagation of zeroes of solutions of partial differential equations},
 fjournal = {Bulletin of the American Mathematical Society},
 journal = {Bull. Am. Math. Soc.},
 issn = {0002-9904},
 volume = {78},
 pages = {835--839},
 year = {1972},
 language = {English},
 doi = {10.1090/S0002-9904-1972-13053-2},
 zbMATH = {3399713},
 Zbl = {0253.35004}
}
\end{document}